\documentclass[10pt,a4paper]{amsart}
\usepackage{amsfonts,amsmath,amssymb}
\usepackage{hyperref}
\usepackage[all]{xy}

\usepackage{amssymb,amsthm,amsxtra}
\usepackage[usenames]{color}
\usepackage[dvipsnames]{xcolor}

\usepackage{amscd}
\usepackage{amsthm}
\usepackage{amsfonts}
\usepackage{amssymb}
\usepackage{mathrsfs}
\usepackage{enumerate}

\newtheorem*{theorem*}{Theorem}
\newtheorem*{remark*}{Remark}
\newtheorem{lemma}{Lemma}[section]

\newtheorem{remark}[lemma]{Remark}
\newtheorem{example}[lemma]{Example}
\newtheorem{theorem}[lemma]{Theorem}

\newtheorem{ques}[lemma]{Question}

\newtheorem*{conjecture*}{Conjecture}

\newtheorem{thm}[lemma]{Theorem}
\newtheorem{prop}[lemma]{Proposition}
\newtheorem{lem}[lemma]{Lemma}
\newtheorem{defn}[lemma]{Definition}
\newtheorem{notn}[lemma]{Notation}
\newtheorem{cor}[lemma]{Corollary}
\newtheorem{exm}[lemma]{Example}

\newtheorem{rem}[lemma]{Remark}

\newtheorem{introtheorem}{Theorem}

\newtheorem{introcorollary}[introtheorem]{Corollary}

\newtheorem{introprop}[introtheorem]{Proposition}

\baselineskip 18pt \textwidth 16.5cm  \theoremstyle{plain}
\newcommand{\Exp}{\operatorname{Exp}}

\newcommand{\bfO}{{\bf O}}

\newcommand{\cC}{\mathcal{C}}

\newcommand{\Hom}{\operatorname{Hom}}

\newcommand{\A}{\mathbb{A}}

\newcommand{\Ind}{\operatorname{Ind}}

\newcommand{\re}{\operatorname{Re}}
\renewcommand{\Im}{\operatorname{Im}}
\newcommand{\Ker}{\operatorname{Ker}}

\newcommand{\R}{{\mathbb R}}

\newcommand{\C}{{\mathbb C}}
\newcommand{\bC}{{\mathbb C}}

\newcommand{\cW}{{\mathcal W}}

\newcommand{\alp}{{\alpha}}

\newcommand{\lam}{{\lambda}}

\newcommand{\Fre}{{Fr\'{e}chet}\,}

\newcommand{\g}{{\mathfrak{g}}}

\newcommand{\Supp}{\mathrm{Supp}}

\newcommand{\cO}{{\mathcal{O}}}
\newcommand{\GL}{\operatorname{GL}}

\newcommand{\Id}{\operatorname{Id}}

\newcommand{\Sc}{{\mathcal S}}

\newcommand{\into}{\hookrightarrow}
\newcommand{\onto}{\twoheadrightarrow}
\newcommand{\hot}{\,\widehat{\otimes}\,}

\newcommand{\cM}{\mathcal{M}}

\newcommand{\fg}{\mathfrak{g}}

\newcommand{\fr}{\mathfrak{r}}
\newcommand{\fp}{\mathfrak{p}}
\newcommand{\fq}{\mathfrak{q}}
\newcommand{\fn}{\mathfrak{n}}

\newcommand{\bP}{\mathbb{P}}

\newcommand{\X}{\mathbf{X}}
\newcommand{\Y}{\mathbf{Y}}
\newcommand{\bfQ}{\mathbf{Q}}
\newcommand{\bfS}{\mathbf{S}}

\newcommand{\W}{\mathbf{W}}
\newcommand{\G}{{\mathbf G}}
\newcommand{\bfH}{{\mathbf H}}
\newcommand{\bfZ}{{\mathbf Z}}

\newcommand{\Dima}[1]{{{#1}}}
\newcommand{\DimaA}[1]{{{#1}}}
\newcommand{\DimaB}[1]{{{#1}}}
\newcommand{\DimaC}[1]{{{#1}}}
\newcommand{\DimaD}[1]{{{#1}}}
\newcommand{\DimaE}[1]{{{#1}}}
\newcommand{\DimaF}[1]{{{#1}}}
\newcommand{\DimaH}[1]{{{#1}}}
\newcommand{\DimaI}[1]{{{#1}}}
\newcommand{\DimaJ}[1]{{{#1}}}
\newcommand{\DimaK}[1]{{{#1}}}
\newcommand{\DimaL}[1]{{{#1}}}
\newcommand{\DimaM}[1]{{{#1}}}
\newcommand{\DimaN}[1]{{{#1}}}

\begin{document}

\author{Dmitry Gourevitch}
\address{ The Incumbent of Dr. A. Edward Friedmann Career Development Chair in Mathematics, Faculty of Mathematics and Computer Science,
Weizmann Institute of Science, 234 Hertzl St,  Rehovot 7610001, Israel }
\email{dimagur@weizmann.ac.il}
\urladdr{http://www.wisdom.weizmann.ac.il/~dimagur}

 \author{Siddhartha Sahi}
 \address{
 Department of Mathematics,
 Rutgers University,
 Hill Center - Busch Campus,
 110 Frelinghuysen Road
 Piscataway, NJ 08854-8019, USA}
 \email{sahi@math.rugers.edu}

 \author{Eitan Sayag}
 \address{
 Department of Mathematics,
Ben Gurion University of the Negev,
P.O.B. 653,
Be'er Sheva 84105,
ISRAEL}
 \email{eitan.sayag@gmail.com}

\date{\today}
\title{Analytic continuation of equivariant distributions}
%
%
%
%
%
%
%
%
%
\begin{abstract}
We establish a method for constructing equivariant distributions on smooth real algebraic  varieties from equivariant distributions on Zariski open subsets. This is based on Bernstein's theory of analytic continuation of holonomic distributions. We use this to construct $H$-equivariant functionals on principal series representations of $G$, where $G$ is a real reductive group and $H$ is an algebraic subgroup. We  also deduce the existence of generalized Whittaker models for degenerate principal series representations.
As a special case, this gives short proofs of existence of Whittaker models on principal series representations, and of analytic continuation of standard intertwining operators.  Finally, we extend our constructions to the $p$-adic case using a recent result of Hong and Sun.
\end{abstract}

\maketitle

\section{Introduction}

The utility of the theory of distributions in representation theory and harmonic analysis is well established since the foundational works of Bruhat and Harish-Chandra. In particular, invariant distributions provide a basic tool to study linear functionals on induced representations, maps between induced representations, and characters of infinite dimensional representations. In many of these applications, the representation theoretic question is translated to a question of existence of certain equivariant distributions on a homogeneous space, usually obtained as points over a local field of an algebraic variety. In some cases one can guess the restriction of the equivariant distribution to an open subset and would like to extend it to the entire space. Although such an extension problem is too general to decide we have found that in many interesting cases, the mere existence of an equivariant distribution on a subvariety implies the existence of such an equivariant  distribution on the entire space.

To illustrate our idea \Dima{let us provide a short proof for the existence of Whittaker models for the principal series representations of $G:=\GL(2,\R)$. Let $B\subset G$ be the Borel subgroup consisting of upper-triangular matrices, $N\subset B$ be the unipotent radical and $S:=B\times N$. Let $\chi$ be a character of $S$ with non-trivial unitary restriction to $\{\Id\}\times N$, \DimaK{and let $\pi$ be the corresponding principal series representation}. Then the existence of Whittaker models \DimaK{for $\pi$} is equivalent to the existence of non-zero distributions on $G$ that are $(S,\chi)$-equivariant, where the action of $S$ is given by left multiplication by $B$ and right multiplication by $N$.
Note that the open Bruhat cell has a (non-zero) $(S,\chi)$-equivariant smooth measure $\xi$. Let $p$ be the \DimaK{non-negative} polynomial on $G$ given by $p(g):=(g_{21})^2$.
Then $p$ is $(S,\psi)$-equivariant where $\psi(b,n)=b_{11}^{-2}b_{22}^2$. Note that $p$ vanishes on $B$ and thus \DimaD{there exists a power $n\geq 0$ such that } $p^n\xi$ extends to an $(S,\psi^n\chi)$-equivariant distribution $\eta$ on $G$. To obtain an $(S,\chi)$-equivariant distribution we consider the family $p^{\lambda-n}\eta$ defined for $\re \lambda$ large enough. By \cite{Ber} (see \S \ref{sec:Dmod} below) this family has a meromorphic continuation to the entire complex plane. The leading coefficient  \DimaK{of the Laurent expansion of this family at zero (see Definition \ref{def:mer} below)} is a non-zero $(S,\chi)$-equivariant distribution on $G$.

More generally, let a linear algebraic group $\bfS$ act on a smooth affine algebraic variety $\X$, both defined over $\R$.} Let $q$ be  a  polynomial on $\X$ transforming under the action of $\bfS$ by some character $\psi$. Define a polynomial $\bar q$ by $\bar q(x):=\overline{q(\bar x)}$. Let $p:=q\bar q$ and let $\Y:=\X_p$ be the basic open subset defined by $p$. Let $X,Y$ and $S$ denote the real points of $\X,\Y$ and $\bfS$.
Let $\xi$ be an $(S,\chi)$-equivariant tempered distribution on $Y$, {\it i.e.} a continuous functional on the space of Schwartz functions on $Y$ (see \S \ref{subsec:Sch} below).
Since $p$ is positive on $Y$, we can consider the product $p^{\lambda} \xi$ as a tempered distribution on $Y$, for any $\lam\in \C$.  Since $p$ vanishes on the complement of $Y$, for $\re \lam $ large enough the new distributions $p^{\lambda} \xi$ naturally extend to the entire space $X$. If $\xi$ generates a holonomic $D(\X)$-module (see \S \ref{sec:Dmod} below for this notion) then \cite{Ber} implies that this family of distributions on $X$ has a meromorphic continuation to the entire complex plane. The obtained family $\eta_{\lam}$ is $(S,\chi |\psi|^{2\lam})$-equivariant.  This implies that the leading coefficient of the family $\eta_{\lam}$ at $0$ is $(S,\chi)$-equivariant, and the constant term $\eta_0$ is \DimaK{(see Definition \ref{def:mer} below)} generalized $(S,\chi)$-equivariant. In addition we have $\eta_0|_Y=\xi$.
For the detailed proof see Lemma \ref{lem:key} below.

Let us provide a class of examples in which for given $\bfS,\X$ and $\Y$, an equivariant polynomial $q$ describing $\Y$ always exists. Namely, we require $\bfS$ to be solvable and assume that it has an  open orbit $\bfO$ on $\X$. Then the Lie-Kolchin theorem implies that the ideal of polynomials vanishing on the complement to $\Y:=\bfO$ has a non-zero $\bfS$-equivariant element $q$.
Note that in this case, the equivariant distribution $\xi$ on $\bfO$ is necessarily a measure.

We show that it suffices to assume the existence of a measure on an arbitrary orbit, not necessary open. \DimaA{We also generalize the argument to quasi-\DimaK{projective} varieties.} Since many invariant distributions arising in representation theory are measures on orbits, we obtain numerous applications, recasting many well known results under one roof.


\begin{introtheorem}\label{thm:A}
Let a \DimaF{Zariski} \DimaD{connected}
solvable linear algebraic group $\bfS$ act on a smooth \DimaA{quasi-\DimaK{projective}} algebraic variety $\X$, both defined over $\R$. Let $S$ and $X$ denote the real points of $\bfS$ and $\X$, and let $\chi$ be a (continuous) character of $S$.
Suppose that some $S$-orbit $O\subset X$  admits a non-zero $(S,\chi)$-equivariant tempered measure.


Then there exists a non-zero $(S,\chi)$-equivariant tempered distribution on $X$.
\end{introtheorem}
Here, a measure is called tempered if it defines a tempered distribution. Note that an $(S,\chi)$-equivariant measure on $O$ is unique up to a multiplicative constant, and it is tempered if and only if the restriction of $\chi$ to the unipotent radical of  $S$ is unitary. See  \S \ref{subsec:Sch} below for more details.

\begin{exm}\label{ex:R}
Let $S:=\R^{\times}$, \DimaK{acting by multiplication on} $X:=\R$ and $O:=\R^{\times}$. Then $O$ carries an $S$-invariant measure $\mu=dx/|x|$, which extends to a generalized-invariant distribution $\xi$ on $X$ given by
$$\xi(f):=\int_{-\infty}^{-1}f(x)|x|^{-1}dx+\int_{-1}^1 (f(x)-f(0))|x|^{-1}dx+ \int_{1}^{\infty}f(x)|x|^{-1}dx.$$
However, $\xi$ is not $S$-invariant. To obtain an $S$-invariant distribution on $X$, we derive $\xi$ and get the delta-distribution  $\delta_0$. This shows that we can either obtain a generalized-invariant extension of the original distribution, or an invariant distribution, but (in general) not an invariant extension.
\end{exm}


To demonstrate the power of Theorem \ref{thm:A},
let $G$ be a quasi-split real reductive group, $B$ be a Borel subgroup, $N$ be the unipotent radical of $B$, $\chi_1$ be a character of $B$ and $\chi_2$ be a non-degenerate unitary character of $N$. Take $S:=B\times N$, consider the two-sided action of $S$ on $G$, let $X:=G$ and $\chi:=\chi_1\otimes \chi_2$. Then Theorem \ref{thm:A} implies the existence of Whittaker models for the principal series representations of $G$.

We prove a generalization of Theorem \ref{thm:A} in \S \ref{sec:main} below, \DimaB{see Theorem \ref{thm:GenA0}}. This generalization allows $S$ to be a \DimaJ{Zariski connected} extension of a compact group by a solvable one. An example of such $S$ is any \DimaJ{Zariski connected} algebraic subgroup of a minimal parabolic subgroup of a real reductive group. \DimaA{We also replace the character $\chi$ by a \DimaK{finite-dimensional} representation of moderate growth \DimaB{(see \S \ref{subsec:ModGr} below).}}


We also extend the result to p-adic local fields, using \cite{HS}.
\DimaC{However, \cite{HS} puts an extra assumption on the representations. In order to shorten the formulations of our main results we introduce the following ad-hoc notation.
\begin{notn}\label{not:C}
If $\bfS$ is defined over an Archimedean field then we denote by
$\cC^{\mathrm{fin}}(S)$ the category of \DimaK{finite-dimensional} representations of moderate growth.  If $\bfS$ is defined over a p-adic local field then we denote by $\cC^{\mathrm{fin}}(S)$ the category of \DimaK{finite-dimensional}  smooth representations that are trivial on the unipotent radical of $S$.
\end{notn}}

\DimaC{In the subsequent discussion we let $F$ be an arbitrary local field of characteristic zero, Archimedean or not.}
Another natural way to generalize Theorem \ref{thm:A} is to consider also generalized sections of bundles on $X$. We do that for the case when $\mathbf{X}$ is a transitive space of a  group  $\mathbf{G}$ that includes $\mathbf{S}$. More specifically, let $F$ be a local field of characteristic zero, and let $\mathbf{G}$ be a \DimaJ{Zariski connected} linear algebraic $F$-group. Let $\mathbf{P}_0\subset \G$ be a minimal parabolic $F$-subgroup. \DimaD{Let $\bfH\subset \G$ and $\bfS\subset \mathbf{P}_0$ be $F$-subgroups and assume that $\bf S$ is Zariski connected}.  Let $H,S,G$ be the groups of $F$-points of $\bfH,\bfS$, and $\G$.
Consider the action of $\bfS\times \bfH$ on $G$ given by left multiplication by $\bfS$ and right multiplication by $\bfH$.
\DimaA{Let $(\sigma,V)\in \cC^{\mathrm{fin}}(S\times H)$.}

\begin{introtheorem}\label{thm:B}

Suppose that some double coset $SgH\subset G$ admits a non-zero tempered $S\times H$-equivariant $\DimaD{V^*}$-valued measure.
Then there exists a non-zero $S\times H$ -equivariant $\DimaD{V^*}$-valued tempered distribution on $G$.
\end{introtheorem}

 In \S \ref{sec:main} below we prove a generalization of Theorem \ref{thm:B} that bounds \DimaK{the dimension of the space of $S\times H$ -equivariant $\DimaD{V^*}$-valued tempered distributions} from below by the number of $(S\times H,V^{*})$-measurable double cosets that lie in the same $\bfS\times \bfH$-double coset (see Theorem \ref{thm:GenB}).  Our method provides no upper bounds on \DimaK{this dimension}.


Taking $\bfS=\bfH =\mathbf{P}_0$ \DimaB{in Theorem \ref{thm:B}} we obtain \DimaC{an alternative proof of} the existence of Knapp-Stein intertwining operators \cite{KnSt}. More generally, by taking  $\bfS=\mathbf{P}_0$ and $\bfH$ arbitrary, Theorem \ref{thm:B} can be used to construct $H$-invariant functionals on principal series representations of $G$.

Namely, assume that $G$ is reductive \DimaD{and Zariski connected,}  \DimaA{let $\sigma\in \cC^{\mathrm{fin}}(P_0)$} and denote by $\Ind_{P_0}^G(\sigma)$  the smooth induction. For any $g\in G$ let $I_g$ denote the group $g^{-1}Hg\cap P_0$, and let $\Delta_{I_g}$ and $\Delta_H$ denote the modular functions of $I_g$ and $H$.
Define a character $\chi_g$ of $I_g$ by $\chi_g(x):=\Delta_{I_g}(x) \Delta_H^{-1}(gxg^{-1})$.

\begin{introcorollary}[\DimaC{See \S \ref{sec:main}}]\label{cor:Prin}
Suppose that for some $g\in G$, $\sigma$ has a \DimaD{non-zero continuous linear} \Dima{functional} that changes under the action of $I_g$ by the character $\chi_g$. Then there exists an $H$-invariant  \DimaD{non-zero continuous linear} functional on $\Ind_{P_0}^G(\sigma)$.
\end{introcorollary}


Another natural question that arises is whether one can extend distributions supported on an orbit, and not just measures defined on an orbit. While we do not have a general result in this direction, in Example \ref{exm:ExtDistQues}
below we show how Theorem \ref{thm:A} can  be used for this purpose.

\DimaB{Another question is whether one can prove an analog of Theorem \ref{thm:B} and Corollary \ref{cor:Prin} in which $P_0$ is replaced by a bigger parabolic subgroup. While we could not yet answer this question in general, \DimaK{we can do so for  $F$-spherical subgroups $H$ and certain  related parabolics. For this purpose we will need a notation.
\begin{notn}\label{not:CW}
For a linear algebraic $F$-group $P$ we set $\cM(P)$ to be the category of representations of $P$ that are trivial on the unipotent radical, and on the reductive quotient of $P$ are
\begin{enumerate}
\item finite-dimensional if $F$ is p-adic, and
\item smooth admissible finitely-generated \Fre representations of moderate growth (a.k.a. Casselman-Wallach representations) if $F$ is Archimedean.
\end{enumerate}
\end{notn}
}
\begin{introprop}[\DimaC{See \S \ref{sec:main}}]\label{prop:ParAdapt}
Let $\mathbf{G}$ be a \DimaD{Zariski connected} reductive algebraic $F$-group, $\mathbf P_0\subset \G$ be a minimal parabolic $F$-subgroup, and $\mathbf{H}$ be an \DimaK{(absolutely) spherical} $F$-subgroup such that $\mathbf{HP_0}$ is open in $\G$. Let  $\mathbf{P}\subset \G$ be a \DimaK{parabolic} $F$-subgroup such that $\mathbf{P_0\subset P}$ and $\mathbf{HP}=\mathbf{HP_0}$.
 Let $\sigma \in \DimaK{\cM}(P)$.
Let $C:=H\cap P$ and suppose that $\sigma$ admits a non-zero $(C,\Delta_C \Delta_H^{-1})$-equivariant continuous \DimaD{linear} functional.
Then there exists an $H$-invariant  continuous \DimaD{non-zero linear} functional on $\Ind_{P}^G(\sigma)$.
\end{introprop}}

An important example of $\mathbf P$ as above is the $\bf G/H$-{\it adapted} parabolic subgroup that includes $\mathbf{P_0}$, see \cite[Definition 2.7 and Theorem 2.8]{KKS}.
For another version of Proposition \ref{prop:ParAdapt} see Proposition \ref{prop:ParSpher} below.

\DimaC{It appears that the methods of this paper allow a construction of invariant linear forms on certain subquotients of the degenerate principal series studied in \cite{Sah}, generalizing \cite{GSS}. We plan to investigate this distinction problem in a forthcoming paper.}


\subsection{Applications to generalized Whittaker models}
Let $G$ be a \DimaD{Zariski connected} real reductive group.
As mentioned above, Theorem \ref{thm:A} implies the well known result regarding the existence of Whittaker models for principal series representations of $G$\ \cite{Jac,Kos}.
More generally, we deduce from Theorem \ref{thm:B}  the non-vanishing of generalized Whittaker spaces for degenerate principal series representations.
In \S \ref{sec:GenWhit} we recall the notion of generalized Whittaker spaces, and  prove the following theorem.

\begin{introtheorem}\label{thm:C}
Let $P\subset G$ be a parabolic subgroup and let $e$ be an element of the nilradical of the Lie algebra of $P$. Let
$V$ be a finite-dimensional (continuous) representation of $P$, and let $\Ind_P^G(V)$ denote the smooth induction of $V$ to $G$.  Then the generalized Whittaker space $\cW_e(\Ind_P^G(V))$ does not vanish.
\end{introtheorem}

Note that if $P$ is a minimal parabolic then Theorem \ref{thm:C} implies the non-vanishing of  $\cW_e(\Ind_P^G(V))$  for all nilpotent $e$ in the Lie algebra of $G$.
\subsection{Related results}
Let $G$ be a real reductive group, and $H$ be a (not necessarily compact) symmetric subgroup.
In \cite{vdB,BD} similar constructions were performed to give $H$-invariant functionals on principal series and generalized principal series representations of $G$.
Recent works  \cite{MOO,MO,Moe} deal with a related problem of constructing symmetry-breaking operators. Namely, they construct $H$-intertwining operators from certain principal and degenerate principal series representations
of $G$ to  certain (degenerate or not) principal series representations of $H$, \DimaN{and show that generically their constructions exhaust the whole space of intertwining operators.}
Note that the space of such intertwining operators (a.k.a. \emph{symmetry breaking operators})
is isomorphic to the space of $\Delta H$-invariant functionals on (degenerate or not)  principal series representations of $G\times H$, where $\Delta H$ denotes the image of the diagonal embedding of $H$ into $G\times H$.
Thus, \DimaN{Corollary \ref{cor:Prin} and Proposition \ref{prop:ParAdapt} extend some results of \cite{vdB,BD,MOO,MO,Moe} by considering functionals invariant under arbitrary algebraic subgroups. On the other hand, \cite{vdB,BD,MOO,MO,Moe}  allow inductions from non-minimal parabolic subgroups.}

The point of departure in the above-mentioned works is the theory of Knapp-Stein intertwining operators, and is thus directly based on analytic considerations.
Our approach is more algebraic, and covers also  spaces  that do not arise from symmetric pairs, in particular non-affine homogeneous spaces.

One can view both \cite{MOO,MO,Moe} and the $\Delta H \subset G\times H$-case of Corollary \ref{cor:Prin} and Proposition \ref{prop:ParAdapt} above as  part of the general program of constructing symmetry breaking operators, see \cite{Kob} and references therein. Some of the operators are constructed in that project  through their kernel distribution (as in the present paper), while some others are given by explicit differential operators.


\DimaA{We would like to mention that \cite{MO} use a different method for analytic continuation - the translation principle, which involves tensoring with finite-dimensional representations of $G$. This method was previously used in \cite{VW} for standard intertwining operators, and in \cite{CD} for symmetric pairs.}

We remark that a special case of our key Lemma \ref{lem:key} was formulated in \cite[Remark 3(ii)]{GSS}.

Our main motivation for the study of generalized Whittaker spaces comes from  \cite{MW}, which characterizes the existence of generalized Whittaker spaces for representations of $p$-adic reductive groups in terms of the wave-front sets of the representations. In \cite{Mat90} a certain partial analog of \cite{MW} is provided for complex reductive groups.
However, for $F=\R$ only very partial analogs of \cite{MW,Mat90} are proven. We view Theorem \ref{thm:C} as one more partial result of this kind, since it establishes the existence of all the generalized Whittaker models for degenerate principal series  that were expected to exist.
\DimaA{Regarding the classical Whittaker models, the classical method gives analytic continuation, while our method a-priory allows poles. However, one can show that actually our method creates no poles. Indeed, on the open Bruhat cell there are no poles, and thus if the family of distributions on the group that we construct would have a pole at some $\lambda_0$ then the principal part of the family would be supported on the complement to the open Bruhat cell. This support condition contradicts the equivariance properties (see  \cite[Lemma 6.5]{CHM} for a stronger statement).}

\DimaA{}

\subsection{Structure of the paper}
In  \S\ref{sec:Prel} we collect some basic results on invariants of quasi-elementary groups and some basic facts about Schwartz \DimaJ{functions and tempered} distributions.

In  \S \ref{sec:Dmod} we collect some facts about holonomic $D$ modules and, following Bernstein,  show their utility in meromorphic continuation of families of distributions.

In  \S \ref{sec:main} we prove a key result concerning extension of equivariant distributions (Lemma \ref{lem:key}) and deduce from it \DimaA{and from \cite{HS}} generalizations of Theorems \ref{thm:A} and  \ref{thm:B}. We also discuss the possibilities of further generalizations. Finally, we deduce Corollary \ref{cor:Prin} from Theorem  \ref{thm:B} \DimaA{and deduce Proposition \ref{prop:ParAdapt}  from Lemma \ref{lem:key}, \cite{HS} \DimaL{and the Casselman embedding theorem}}.

In \S \ref{sec:GenWhit} we recall the notion of generalized Whittaker spaces and prove a generalization of Theorem \ref{thm:C}.



\subsection{Acknowledgements}
We thank Avraham Aizenbud, Joseph Bernstein, Michele Brion, Shachar Carmeli, \DimaL{Dmitry Faifman,} Herve Jacquet, \DimaA{David Kazhdan,} Bernhard Kroetz, Andrey Minchenko, \DimaA{and Dmitry Timashev} for fruitful discussions\DimaD{, and the anonymous referees for the careful proofreading and useful remarks}.

D.G. was partially supported by ERC StG grant 637912, and ISF grant 756/12.
E.S. was partially supported by ERC grant 291612.
\Dima{Part of the work was done while D.G. was visiting the Institute for Mathematical Sciences, National University of Singapore in 2016. The visit was supported by the Institute.}
\DimaC{Part of the work was done while E.S. was visiting the Max-Planck-Institut f{\"u}r Mathematik in Bonn. He wishes to thank the MPIM for excellent working conditions.}
\DimaB{Part of the work was done while S. Sahi was a
Rosi and Max Varon Visiting Professor at the Weizmann Institute of Science. }

\section{Preliminaries}\label{sec:Prel}
\DimaH{Throughout the paper we let} $F$ be a local field of characteristic zero. We will denote algebraic varieties and algebraic groups defined over $F$ by bold letters and their $F$-points by the corresponding letters in regular font.

For a representation $V$ of a group $G$ we denote by $V^G$ the space of invariants.
\DimaD{We will say that a polynomial $p$ on $V$ is \emph{$G$-equivariant} if it changes by a character under the action of $G$ on the argument. }
\begin{defn}\label{def:GenInv} Let $(\pi,V)$ be a representation of a group G. A vector
$v \in V$ is called a generalized invariant vector if there is a natural number $k$ such that
$$(\pi(g_0) - \Id)(\pi(g_1) -\Id) \cdots (\pi(g_k) -\Id)v = 0  \quad \forall \,g_0, g_1,\dots, g_k \in G.$$
\end{defn}
Note that if $G$ is a connected Lie group and $\pi$ is a smooth representation then $v$ is generalized invariant if and only if the Lie algebra of $G$ acts nilpotently on the subrepresentation generated by $v$.

\subsection{Meromorphic families}
Let   $\C((\lambda))$ denote the field of Laurent power series, and let $E$ be a complex vector space and let $E((\lambda))$ denote the space of Laurent power series with coefficients in $E$. For any real $a>0$ define $a^{\lam}:=\sum_{i\geq 0}(\ln a)^i/(i!)\lambda^i\in \C((\lambda))$.

\begin{lem}\label{lem:terms}
Let a group $G$ act on $E$ linearly, and let $\psi$ be a character of $G$. Extend this action to $E((\lambda))$ in the natural way. Let $$f=\sum_{i=-n}^{\infty} a_i\lambda^i\in E((\lambda)) \text{ satisfy }gf=|\psi(g)|^{\lam}f.$$
Then $a_{-n}$ is $G$-invariant. Moreover, $a_{-n+l}$ is generalized $G$-invariant for any $l\geq 0$.
\end{lem}
\begin{proof}
From $gf=|\psi(g)|^{\lam}f$, comparing term by term,
 we obtain
\begin{equation}
ga_i=\chi(g)\sum_{j=-n}^i \frac{\ln |\psi(g)|^{i-j}}{(i-j)!} a_j
\end{equation}
This implies that $a_{-n}$ is $G$-invariant. The ``moreover" part follows by induction on $l$.
\end{proof}

\begin{defn}\label{def:mer}
For $f=\sum_{i=-n}^{\infty} a_i\lambda^i\in E((\lambda))$ with $a_{-n}\neq 0$ we say that $a_{-n}$ is the leading coefficient and $a_0$ is the constant term. \DimaD{For convenience we will say that both the leading coefficient and the constant term of the zero series are zero.}
\end{defn}

\begin{lem}\label{lem:Spec}

 Let $L\subset E((\lambda))$ be a finite-dimensional $\C((\lambda))$-vector space, and $W\subset E$ be the $\C$-subspace given by the leading coefficients of the series in $L$. Then $$\dim_{\C} W = \dim_{\C((\lambda))} L.$$
\end{lem}
\begin{proof}
\DimaK{
First of all, note that $W$ is finite-dimensional. Indeed, for any linearly independent set of vectors in $W$, the corresponding set of constant series in $L$ is linearly independent over $\C((\lambda))$. Thus $\dim_{\C} W \leq \dim_{\C((\lambda))} L.$

Now let $\{w_j\}_{j=1}^n$ be a basis for $W$.  For each $j$ we can choose a series $u_j$ in $L$ without negative terms and with $w_j$ as its constant term. It is enough to show that $\{u_j\}_{j=1}^n$ span $L.$ Clearly, it is enough to show that any $u \in L$ whose leading term is the constant term is a  $\C((\lam))-$combination of the ${u_j}.$ To show that, let  $c_j \in \C$ be the unique constants such that $u'=u -\sum _{j=1}^nc_j u_j$ has constant term 0.
Next let $v' = u'/ \lambda$, then there exist unique coefficients $c^1_j$ such that $u''=v'-\sum _{j=1}^nc^1_ju_j$ has constant term 0. Proceeding in this way by induction we express $u=\sum _{j=1}^na_j u_j$ where $a_j = c_j+c^1_j\lambda+c^2_j \lam^2+\cdots...$}
\end{proof}

\subsection{\DimaA{Actions of $F$-groups}}\label{subsec:GrAct}
%
\begin{notn}
Following \cite{KK}, an $F$-variety for us is a variety over $\bar F$ equipped with a compatible action of the Galois group $\mathrm{Gal}(\bar F /F)$. Accordingly, for us the terms ``connected" and ``irreducible" mean ``absolutely
connected" and ``absolutely irreducible".
\end{notn}

\begin{prop}[{\cite[\DimaD{\S 1.A.2 } and Proposition 2.6]{Pop}}]\label{prop:large}
Let $\bf X$ be an irreducible $F$-variety that has a smooth $F$-point. Then $\bf X$ is $F$-dense, i.e. the set of $F$-points in $\bf X$ is (Zariski) dense in $\bf X$. \DimaH{In particular, any connected homogeneous $F$-variety with an $F$-point is $F$-dense.}
\end{prop}

Let $\bfQ$ be a linear algebraic $F$-group. Following \cite[\S 3]{KK} we say that $\bfQ$ is quasi-elementary if it does not contain a proper parabolic $F$-subgroup. Actually, \cite[\S 3]{KK} give a different definition, but explain why it is equivalent to the one given here. \DimaD{Note that a quasi-elementary group is necessarily connected}.

Note that all \DimaD{the connected} solvable groups are quasi-elementary and if $F=\R$ then so are all the \DimaD{connected} \DimaF{linear algebraic $\R$-}groups with compact-by-solvable $\R$-points.

\begin{prop}\label{prop:KK}
\DimaB{Let
a quasi-elementary linear algebraic $F$-group $\bfQ$ act on a quasi-affine irreducible $F$-dense  $F$-variety $\bf X$.} Let $I \subset  \cO(\bf X)$ be a non-zero $\bf Q$-stable ideal.  Then $I$
contains a non-zero $\bf Q$-equivariant \DimaA{regular function} defined over $F$.
\end{prop}
\begin{proof}
For affine $\X$ this is {\cite[Proposition 3.10]{KK}}.

Let us show that any quasi-affine $\X$ can be $\bf Q$-equivariantly embedded into an affine
$F$-variety $\bf Y$ as an open dense subset.
Indeed,  there exists a generating ${\bf Q}$-invariant finite dimensional subspace $\W\subset \cO(\X)$. Such $\W$ defines an embedding of $\X$ into $\W^*$. Let $\bf Y$ be the closure of $\X$ in $\W^*$.

Now, let $J$ be the preimage of $I$ under the restriction morphism $\cO(\Y)\into \cO(X)$. Since $\Y$ is affine and $J$ is $\bf Q$-stable, $J$ contains a non-zero $\bf Q$-equivariant polynomial $p$ defined over $F$. Restricting $p$ to $\X$ we obtain the required regular function.
\end{proof}

\begin{cor}\label{cor:pExists}
\DimaB{Let
a quasi-elementary linear algebraic $F$-group $\bfQ$ act on a quasi-affine irreducible $F$-variety $\bf X$. Suppose that $\bf X$ has an open orbit $\bfO$ that has an $F$-point.} Then \(\mathbf{O}\)  is a  basic open subset of $\mathbf{X}$ defined by a $\mathbf{Q}$-equivariant polynomial defined over $F$. \end{cor}
\begin{proof}
\DimaA{
By Proposition \ref{prop:large}  $\bf X$ is $F$-dense.
Let $I$ be the ideal defining $\mathbf{X} \setminus \mathbf{O}$. By Proposition \ref{prop:KK}, $I$ contains a non-zero $\mathbf{Q}$-equivariant polynomial $p$. Then $p(x)\neq 0$ for some $x\in {\bf O}$ and thus for all $x\in {\bf O}$.}
\end{proof}

\begin{exm}\label{exm:grpCase}
\DimaC{
Let $\bf G$ be a reductive group, ${\bf P \subset G}$ be a parabolic  subgroup, and ${\bf \bar P}$ be an opposite parabolic ({\it i.e.} $\bf L:=P\cap\bar P$ is a Levi subgroup of both $\bf P$ and $\bf \bar P)$.  Suppose that $\bf G,P$ and $\bf \bar P$ are defined over $F$. Consider the two-sided action of ${\bf Q}:={\bf P\times \bar P}$ on $\X:=\bf G$. Let us describe the polynomial $p$  that defines the open Bruhat cell $\bf P\bar P$ in $\bf G$.

Let $\fn$ and $\bar \fn$ be the nilradicals of the Lie algebras of ${\bf P}$ and $\bf \bar P$ in correspondence.  Let $d:=\dim \fn =\dim \bar\fn$.
Let $\g$ denote the Lie algebra of $\bf G$. Let $v\in \Lambda^d(\fn)\subset \Lambda^d(\g)$ and $\bar v\in \Lambda^d(\bar \fn)\subset \Lambda^d(\g)$ be non-zero vectors. Let $(\sigma,V)$ and $(\bar \sigma, \bar V)$ be the subrepresentations of $ \Lambda^d(\g)$ generated by $v$ and $\bar v$. Then $V$ and $\bar V$ are irreducible and contragredient to each other. They define a matrix coefficient function $p(g):=\langle \bar v, \sigma(g) v \rangle$.
Let us show that the zero set of $p$ is precisely $\bf G - P\bar P$.

Choose a maximal torus $\bf T$ and a Borel subgroup $\bf B$ in $\bf G$ such that $\bf T \subset B \subset P$ and $\bf T\subset \bar P$. By the Bruhat decomposition, one can choose a subset $W'$ of the normalizer of $T$ that includes the unit and is a full set of representatives for the double coset space $\bf P\setminus G /\bar P$.
 Note that $v$ is the highest weight vector of $V$ and $\bar v$ is the lowest weight vector of $\bar V$. Note also that all the elements of $W'$ except the unit move $v$ to a weight vector of a different weight, which thus pair to zero  with $\bar v$. Therefore $p$ vanishes on all the elements of $W'$ except the unit. Since $p$ is $\bf Q$-equivariant, and $p(1)\neq 0$, it satisfies the conditions. }
\end{exm}

\begin{cor}\label{cor:pQH}
Let $\G$ be a linear algebraic $F$-group and let $\bfQ,\bfH$ be $F$-subgroups, where $\bfQ$ is quasi-elementary. \DimaH{Let $\bfQ\times \bfH$ act on $\G$ by left/right multiplication and let $\bfQ g \bfH$ be a double coset of some $F$-point $g\in G$. Let $\bfZ$ denote the Zariski closure of $\bfQ g \bfH$.

\DimaI{Then there exists a non-zero $\bfQ \times \bfH$-equivariant polynomial $q$  on $\bf Z$ defined over $F$ and vanishing outside $\bfQ g\bfH$.}}
\end{cor}
\begin{proof}
By Chevalley's theorem (see e.g. \cite[Chapter II, Theorem 5.1]{BorGr}) there exist an algebraic representation $\W$ of $\G$ and a line $\bf D \subset W$ both defined over $F$ such that $\bfH=\{x\in \G \, \vert \, x{\bf D=D}\}$.

\DimaD{Let  $\X$ denote the closure of $\bfQ g{\bf D}$ in $\W$.  Denote ${\bf D^\times:=D}\setminus \{0\}$ and note that $\bfQ g {\bf D^{\times}}$ is open in $\X$ and let $I$ be the ideal of all polynomials on $\X$ that vanish outside $\bfQ g {\bf D^{\times}}$. \DimaH{By Proposition \ref{prop:large}, $\bfQ g {\bf D^{\times}}$ is $F$-dense, and thus so is $\X$.}
 By Proposition \ref{prop:KK}  $I$ contains a non-zero $\mathbf{Q}$-equivariant polynomial $p'$. Let $p$ be the leading homogeneous term of $p'$. Note that $p$ is $\mathbf{Q}$-equivariant as well and $p\in I$.
Let $v\in {\bf D^{\times}}$ be an $F$-vector and define a map $a:\G\to \W$ by $a(x):=xv$. Then $a^{-1}(\bfQ g {\bf D^{\times}})=\bfQ g \bfH$ and thus $a(\bfZ)\subset\X.$
Define $q$ on  $\bfZ$ by $q:=p\circ a$. Note that $q$ is non-zero, $\bfQ \times \bfH$-equivariant and vanishes outside $\bfQ g\bfH$.}
\end{proof}

\DimaA{


\begin{prop}[{\cite[Proposition 1.1]{KKV}}]\label{prop:Invert}
Let $\X$ and $\Y$ be irreducible algebraic varieties, and let $\cO^{\times}(\X)$ and $\cO^{\times}(\Y)$ denote the groups of nowhere vanishing regular functions on them. Then the natural map
$$\cO^{\times}(\X)\otimes \cO^{\times}(\Y)\to \cO^{\times}(\X\times \Y)$$
is an epimorphism.
\end{prop}

\begin{cor}\label{cor:act}
Let a connected linear algebraic group $\mathbf{G}$ act on an irreducible algebraic variety  $\mathbf{\X}$. Let $p\in \cO(\X)$ be a regular function such that the zero set of $p$ is $\G$-invariant. Then $p$ changes under the action of $\G$ by some algebraic character.
\end{cor}
\begin{proof}
Note that the divisor corresponding to $p$ is $\G$-invariant, and thus the ratio $p^g/p$ is a nowhere vanishing regular function on $\X$, where $g\in \G$ and $p^g$ denotes the $g$-shift of $p$. Varying $g$ we obtain a function $f\in \cO^{\times}(\G \times \X)$. By Proposition \ref{prop:Invert}, there exist $a \in \cO^{\times}(\G)$ and $b\in \cO^{\times}(\X)$ such that $f(g,x)=a(g)b(x)$.
Substituting the neutral element of $\G$ we get that $b$ is a scalar. We can choose this scalar to be 1. Note that that $a$ becomes an algebraic character of $\G$ and that $p$ changes under the action of $\G$ by this character.
\end{proof}
}

\DimaN{
\begin{prop}[\cite{KKLV}]\label{prop:QP}
Let a connected linear algebraic $F$-group $\mathbf{G}$ act on a normal quasi-projective $F$-variety  $\mathbf{X}$.
Then there exists an algebraic (finite-dimensional) $F$-representation $\mathbf{W}$ of $\mathbf{G}$ and a (locally closed) $\G$-equivariant embedding of $\X$ into the projective space $\mathbb{P}(\W)$.
\end{prop}
The proof in \cite{KKLV} is given for  algebraically closed fields, but the argument works over
 any field of zero characteristic.}

\subsection{Equivariant distributions on $l$-spaces}\label{subsec:pAdic}
For non-Archimedean $F$ we will consider distributions on  $l$-spaces, {\it i.e.} locally compact totally disconnected Hausdorff topological spaces. This generality includes $F$-points of algebraic varieties defined over $F$ (see \cite{BZ}). For an $l$-space $X$, the space $\Sc(X)$ of test functions consists of locally constant compactly supported functions and the space of distributions is $\Sc^*(X)$,  the full linear dual. All distributions on $l$-spaces are tempered.
In this subsection we assume that $F$ is non-Archimedean and let $\bf G$ be a linear algebraic group defined over $F$. Let $\chi$ be a continuous character of $G$. A generalized $\chi$-equivariant distribution on X is defined to be a generalized invariant
vector in the representation $\Hom_{\C}(\Sc(X), \chi)$ of $G$. Our main tool in the  non-Archimedean case is the following special case of \cite[Theorem 1.5]{HS}.

\begin{thm}\label{thm:padic}
Let $\X$ be an algebraic
variety defined over $F$ such that $\bf G$ acts algebraically on $\X$ with an open orbit $\bf U \subset \X$. Assume
that there is a $(\G,\psi)$-equivariant regular function $f$ on $\X$ (for some character $\psi$ of $\G$) such that $${\bf U}=\X_f=\{x\in \X \text{ with }f(x)\neq 0\}.$$ Assume also that $\chi$ is trivial on the $F$-points of the unipotent
radical of $\G$. Then every generalized $\chi$-equivariant distribution $\xi$ on ${\bf U}(F)$ extends to a
generalized $\chi$-equivariant distribution on $\X(F)$.
Moreover, there exists a meromorphic family $\eta_{\lam}$ of $(G,\chi|\psi|^{\lam})$-equivariant distributions such that the constant term $\eta_0$ of this family at 0 extends $\xi$.
\end{thm}

The ``moreover" part is not formulated in \cite[Theorem 1.5]{HS} but rather follows from the proof. More precisely,  it follows from \cite[Propositions 5.20 and 6.22]{HS}.

\DimaJ{
\subsection{\Fre Spaces, their duals and tensor products}
All the topological vector spaces considered in this paper will be either \Fre  or dual \Fre. For a \Fre  space $V$, $V^*$ will denote the strong dual, and for a \Fre or dual \Fre space $W$, $V\hot W$ will denote the completed projective tensor product and $L_b(V,W)$ will denote the space of bounded linear operators from $V$ to $W$ (see \cite[\S 32]{Tre}). The projective topology on $V\hot W$ is generated by seminorms which are largest cross-norms corresponding to pairs of generating semi-norms on $V$ and $W$, see \cite[\S 43]{Tre}. In particular, if $V$ and $W$ are \Fre  spaces, then so is $V\hot W$. If $V$ (or $W$) is nuclear then the projective tensor product is naturally isomorphic to the injective one, see \cite[Theorem 50.1]{Tre}. This is the case in all our theorems. We will need the next proposition, which follows from \cite[Proposition 50.5 and (50.19)]{Tre}
\begin{prop}\label{prop:TensorDual}
Let $V$ and $W$ be \Fre  spaces, with $V$ nuclear. Then we have natural isomorphisms
$$(V\hot W)^*\cong V^*\hot W^*\cong L_b(V,W^{*}).$$
\end{prop}
}
\DimaM{
We will also use the following lemma.
\begin{lem}[{\cite[Lemma A.3]{CHM}}]\label{lem:flat}
Let $W$ be a nuclear \Fre space and let
$$0\to V_1\to V_2\to V_3\to 0$$
be a short exact sequence of  nuclear \Fre spaces. Then the sequence
$$0\to V_1\hot W\to V_2 \hot W \to V_3\hot W\to 0$$
is also a short exact sequence.
\end{lem}
}


\subsection{Schwartz functions and tempered distributions on real algebraic manifolds}\label{subsec:Sch}
Let ${\bf X}$ be an  algebraic manifold ({\it i.e.} smooth algebraic variety) defined over $\R$ and $X:={\bf X}(\R)$. If $\bf X$ is affine then the \Fre  space $\Sc(X)$ of
Schwartz functions on $X$ consists of smooth complex valued functions that decay, together with all their derivatives,
faster than any polynomial. This is a \DimaD{nuclear} \Fre  space, with the topology given by the system of semi-norms $|\phi|_{d}:=\max_{x\in X}|df|$, where $d$ runs through  all differential operators on $X$ with polynomial coefficients.

For a Zariski open affine subset ${\bf U}\subset {\bf X}$, the extension by zero of a Schwartz function on $U$ is a Schwartz function on $X$. This enables to define the Schwartz space on any algebraic manifold ${\bf X}$, as the sum of the Schwartz spaces of the open affine pieces, extended by zero to functions on $X$.
For the precise definition of this notion see e.g. \cite{AGSc}. Elements of the dual space $\Sc^*(X)$ are called tempered distributions. The spaces $\Sc^*(U)$ for all Zariski open ${\bf U}\subset {\bf X}$ form a sheaf. We say that a measure is tempered if it defines a tempered distribution.

\DimaB{In a similar way one can define the space $\Sc(X,V)$ of $V$-valued Schwartz functions for any  \Fre space $V$.
Namely, for an affine $X$ we demand that $q(d\phi(x))$ is  bounded for any differential operator  $d$ on $X$ and any seminorm $q$ on $V$.
\DimaD{It is easy to see that 
$\Sc(X,V)\cong \Sc(X)\hot V$, where $\hot$ denotes the completed \DimaI{projective} tensor product (cf. \cite[Theorem 51.6]{Tre}}).
We  define the  tempered distributions $\Sc^*(X,V)$ to be the continuous linear dual space.
\DimaJ{Note that by Proposition \ref{prop:TensorDual} we have
\begin{equation}\label{=dist}
\Sc^*(X,V)\cong \Sc^*(X)\hot V^*\cong L_b(\Sc(X),V^*)
\end{equation}

\begin{lem}\label{lem:Plam}
Let $p$ be a polynomial on $\X$ with real coefficients that takes non-negative values on $X$. For any $\lam \in \C$ with $\re \lam >0$ consider the function $p^{\lam}(x)=p(x)^{\lam}$ on $X$.
Let $V$ be a \Fre  space and let $\xi \in \Sc^*(X,V)$.
Then there exists a natural number $N$ such that for any $\lam$ with $\re\lam>N$  the distribution $\xi p^{\lam}$ is well defined.
\end{lem}
\begin{proof}
First of all, we can assume that $\X$ is affine.
By definition, $\xi$ is a functional on $\Sc(X)\hot V$ bounded by a finite combination of semi-norms of the form $q_i\times r_i$, where $r_i$ are semi-norms on $V$ and  $q_i$ are semi-norms on $\Sc(X)$ given by $q_i(f)= \max_{x\in X}|d_i(f)|$, where $d_i$ are some differential operators on $\X$. We let $N$\ be the maximum of the orders of $d_i$. Then for any $\lam$ with $\re \lam >N$ the function $p^{\lam}$ has continuous derivatives up to order $N$, and thus $\xi p^{\lam}$ is well defined.
\end{proof}

}
If a group $G$ acts on $X$ and on $V$ then we consider the diagonal action on $\Sc(X,V)$ and the dual action on $\Sc^*(X,V)$. We denote the space of invariants in $\Sc^*(X,V)$ by $\Sc^*(X,V)^G$. We call the elements of this space \emph{equivariant distributions}.}


Let ${\bf U}\subset {\bf X}$ be a Zariski open subset\DimaI{, write $U:={\bf U}\cap X$} and let $Z$ denote the complement to $U$ in $X$.

\begin{thm}[{\cite[ Theorem 4.6.1 and \S 5.3]{AGSc}}]
\DimaD{Assume $Z$ is smooth. Then }the restriction to $Z$ defines an epimorphism $\Sc(X)\onto \Sc(Z)$.
\end{thm}
Dualizing the map $\Sc(X)\onto \Sc(Z)$ we obtain an embedding $\Sc^*(Z)\into \Sc^*(X)$. We call this map extension of distributions by zero.

\begin{thm}[{\cite[ Theorem 5.4.3]{AGSc}}] \label{pOpenSet}
We have
$$\Sc(U) \cong \{\phi \in \Sc(X)| \quad \phi \text{ is 0 on } Z \text{ with all derivatives} \}.$$
In particular, extension by zero defines a closed imbedding $\Sc(U) \into \Sc(X)$.
\end{thm}

\begin{cor}
The restriction map $\Sc^*(X)\to \Sc^*(U)$ is onto.
\end{cor}

\begin{rem}
Note that this corollary does not hold for arbitrary distributions. For example, the distribution $e^xdx$ does not extend from $\R$ to $\R P^1$. Indeed,  since $\R P^1$ is compact, any distribution on it is tempered and therefore restricts to a tempered distribution on $\R$.
\end{rem}

Let us record one more corollary of this theorem.

\begin{cor}\label{cor:ext}
Let $\bf G$ be an algebraic group defined over $\R$,  and $V$ be a
\DimaD{nuclear}
 \Fre space with a continuous \DimaD{linear} action  of $G:=\bf G(\R)$.
\DimaB{Let $V^*$ denote the continuous linear dual.}
Let $\bf G$ act on $\bf X$ algebraically, let $\bf U \subset \bf X$  be a $\bf G$-invariant Zariski open subset, \DimaI{write $U:={\bf U} \cap X$ and let $Z$ denote the complement to $U$ in $X$.} Let  $\xi\in \Sc^*(U,V)^{G}$. Then for some natural number $n$, $\xi$ extends to a $G$-intertwining operator $\xi':F^n(X,U) \to V^\DimaA{*}$, where $F^n(X,U)$ is the space of Schwartz functions on $X$ that vanish on $Z$ with first $n$ derivatives.

Moreover, there exist $(n,\xi')$ as above such that for any differential operator $d$ on $\bf X$ with polynomial coefficients satisfying $\xi\circ d =0, \, \xi'$
vanishes on $d(F^{n+\deg d}(X,U))$, where $\deg d$ denotes the degree of $d$.
\end{cor}
\begin{proof}
We can assume that $\bf X$ is affine. Then the \Fre space $\Sc(X)$ is the inverse limit of the \Fre spaces $\Sc^n(X)$ consisting of $n$ times differentiable functions that decay rapidly at infinity together with their first $n$ derivatives. Note that $\Sc(X)$ is dense in each of these spaces. Denote by $\Sc^n(X,U)$ the closed subspace of $\Sc^n(X)$ consisting of functions that vanish  on $Z$ with first $n$ derivatives. Then both $F^n(X,U)$ and $\Sc(U)$ are dense in $\Sc^n(X,U)$, and by Theorem \ref{pOpenSet}
\DimaD{we have $\Sc(U)=\cap_n \Sc^n(X,U)$.
Let us show that $\xi$ extends to a continuous linear functional on $\Sc^n(X,U)\hot V$ for some $n$. Indeed, $\xi:\Sc(U)\hot V \to \C$ is continuous and thus bounded by a finite linear combination of semi-norms. By
definition of the projective topology on the tensor product these semi-norms can be chosen to be of the form $p_i\otimes q_i$ where $p_i$ are semi-norms on $\Sc(U)$ and $q_i$ are semi-norms on $V$. Then each $p_i$ can be extended to $\Sc^{n_i}(X,U)$ for some $n_i$. Let $n:=\max_i n_i$.
Then $\xi$ extends to a continuous linear functional on  $\Sc^n(X,U)\hot V$.} \DimaJ{By Proposition \ref{prop:TensorDual}, this functional defines a bounded linear operator $\tilde \xi:\Sc^n(X,U)\to V^*$.}

 Note that the action of $G$ preserves all the spaces mentioned above. Thus the equivariance of $\tilde \xi$ follows from the equivariance of $\xi$ and the density of $\Sc(U)$ in $\Sc^n(X,U)$.
To obtain $\xi'$ we restrict $\tilde \xi$ to  $F^n(X,U)$.
To prove the moreover part, note that $\xi(d(\Sc(U)))=0$ and the density of $\Sc(U)$ in $\Sc^{n+\deg d}(X,U)$ implies that $\tilde \xi (d(\Sc^{n+\deg d}(X,U)))=0$. Restricting to $F^{n+\deg d}(X,U)$ we deduce the vanishing of $\xi' \circ d$ on this space.
\end{proof}

\begin{rem}
More generally, one can define Schwartz sections of Nash (=smooth semi-algebraic) bundles on  Nash manifolds. Theorem \ref{pOpenSet} is proven in \cite{AGSc} in this generality, and Corollary \ref{cor:ext} stays true in this wider generality with identical proof.
\end{rem}

\subsubsection{Representations of moderate growth}\label{subsec:ModGr}

\DimaA{
Let $\bf G$ be an affine algebraic group defined over $\R$ and $G:=\G(\R)$.
A representation $\sigma$ of $G$ in a \Fre space $V$ has \emph{moderate growth} if for any semi-norm $q$ on $V$ there exists a semi-norm $q'$ on $V$ and a regular function $a$ on $G$  such that for any $v\in V$ and $g\in G$ we have
$$q(\sigma(g)v)\leq |a(g)|q'(v).$$

\begin{notn}\label{not:C2}
We denote by $\cC(G)$ the category of smooth   representations of moderate growth, in nuclear \Fre spaces.\DimaK{ As in Notation \ref{not:C} we also denote by $\cC^{\mathrm{fin}}(G)\subset \cC(G)$ the subcategory of finite-dimensional representations.}
\end{notn}

 Note that all unitary characters have moderate growth. If $\bf G$ is a unipotent group, then a
character $\chi$ has moderate growth if and only if $\chi$ is unitary. If $\bf G$ is reductive, then
all continuous characters have moderate growth. These statements reduce to the case of one-dimensional groups, and in this case they are straightforward.
We will need the following special case of a version of Frobenius reciprocity.

\begin{lem}[{\cite[Definition 2.3.1 and Lemma 2.3.4]{GGS}}]\label{lem:Frob}
Let $V\in \cC(G)$.
Let $\bf H\subset \G$ be a Zariski closed subgroup. Let $\Delta_G$ and $\Delta_H$ denote the modular functions of $G$ and $H$.
Then we have a natural isomorphism
\DimaK{between $(V^*)^{H,\Delta_H\Delta_G^{-1}}$ and the space of $G$-invariant $V^*$-valued measures on $G/H$.}
\end{lem}}


For the next two lemmas assume that $\G$ is reductive and let $P\subset G$ be a parabolic subgroup. Let $N$ be the unipotent radical of $P$ and $\fn$ be the Lie algebra of $N$.

\begin{lem}\label{lem:NilpTem}
Let $V$ be a (continuous)  finite-dimensional representation of $P$. Then
\begin{enumerate}[(i)]
\item \label{it:nilp} The action of $\fn$ on $V$ is nilpotent.
\item \label{it:tem} $V$ has moderate growth.
\end{enumerate}
\end{lem}
\begin{proof}
\eqref{it:nilp}
Note that there exists a hyperbolic semi-simple $S\in \fp$ such that
$\fp=\bigoplus_{\lambda\geq 0}\fg(\lambda)$, where $\fg(\lambda)$ denotes the $\lambda$-eigenspace of the adjoint action of $S$. Then $\fn=\bigoplus_{\lambda> 0}\fg(\lambda)$. Decomposing $V$ to generalized eigenspaces of $S$ and using the finiteness of the dimension we obtain that $\fn$ acts nilpotently.

\eqref{it:tem} We can assume that $V$ is irreducible. By  \eqref{it:nilp} this implies that $N$ acts trivially on $V$, and thus the reductive quotient $P/N$ acts on $V$. Thus $V=W\otimes \chi$, where $W$ is an algebraic representation of $P/N$ and $\chi$ is a character of $P/N$. Since both $W$ and $\chi$ have moderate growth, so does $V$.
\end{proof}

\DimaA{
Let $V\in \cC(P)$ and let $\Ind_P^G(V)$ denote the smooth induction.
\begin{lem}[{\cite[Definition 2.1.2 and Remark 2.1.4]{dCl}}] \label{lem:Integ}
 Let $(\Ind_P^G(V))^*$ denote the continuous linear dual. Then we have  natural isomorphisms of
 $G$ - representations
$$(\Ind_P^G(V))^* \cong (C^{\infty}(G, V)^{P})^* \cong \Sc^*(G,\DimaA{V}\otimes   \Delta_P^{-1})^P,$$
where $G$ acts on $C^{\infty}(G, V)$ and on $\Sc^*(G,\DimaA{V}\otimes   \Delta_P^{-1})$ from the left and $P$ acts from the right.
\end{lem}}

\DimaK{
\subsection{Distributions on projective varieties}
\begin{lemma}\label{lem:PrFrob}
Let $\bf G$ be an algebraic group defined over $F$, let $\bf W$ be an algebraic (finite-dimensional) $F$-representation of $\mathbf{G}$, and let $pr:\W^{\times}\to \bP(\W)$ denote the natural projection, where $\W^{\times}$ denotes the complement to $0$ in $\W$ and $\bP(\W)$ denotes the projective space.
Let $\X\subset \bP(\W)$ denote a locally closed subvariety and let $X':=pr^{-1}(X)$. If $F$ is Archimedean assume that $X$ is smooth.
Let $(\sigma,V)\in \cC(G)$.

Then we have an isomorphism of $G$-representations
\begin{equation}\label{=Frob}
\Sc^*(X,V)\simeq \Sc^*(X',V)^{F^{\times}},
\end{equation}
where the multiplicative group $F^{\times}$ acts on $\Sc^*(X',V)$ by the scaling action on $\Sc(X)$.

Moreover, this family of isomorphisms is compatible with restrictions to open subvarieties and with continuation by zero from closed subvarieties, and preserves holonomicity.
\end{lemma}
\begin{proof}
First, fix a Haar measure $\mu$ on $F^{\times}$. Then, for every one-dimensional subspace $\ell\subset W$ we fix an $F^{\times}$-invariant measure $\mu_{\ell}$ on $\ell^{\times}$ to be the pullback of $\mu$ by some linear isomorphism $\varphi:\ell\simeq F$. Note that $\mu_{\ell}$ does not depend on the choice of $\varphi$, since any two choices of $\varphi$ differ by a multiplicative constant, and $\mu$ is invariant to multiplication. Now, define an epimorphism $\alp: \Sc(X')\onto \Sc(X)$ by $$\alp(f)(\ell):=\int_{\ell}f \mu_{\ell}.$$
The isomorphism \eqref{=Frob} is given by the composition of operators $\Sc(X)\to V^*$ with $\alp$. It is easy to see that it has the stated functoriality properties.
\end{proof}

\subsection{The Casselman-Wallach category and the Casselman embedding theorem}

Let $\bf P$ be a linear algebraic $\R$-group and $\bf U$ be its unipotent radical. As in Notation \ref{not:CW} let $\cM(P)$ denote the category of
 smooth finitely generated \Fre representations of moderate growth of $P$ that are trivial on $U$ and define admissible representations of $P/U$ (see \cite[\S 11.5]{Wal2} or \cite{Cas} for the notion of admissible).

We will need the following version of Casselman's theorem.

\begin{thm}\label{thm:CasSubRep}
Let $\pi\in \cM(P)$ and let $Q$ be a minimal parabolic subgroup of $P$. Then there exists a finite-dimensional (smooth)
representation $\sigma$ of $Q$ and an epimorphism $Ind_{Q}^P(\sigma)\onto \pi$.
\end{thm}

This theorem immediately reduces to the case of trivial $U$. This case in turn follows from the Casselman-Wallach equivalence of categories (see \cite[\S 11.6.8]{Wal2} or \cite{Cas}), the Casselman embedding theorem \cite[Proposition
8.23]{CM} and the standard duality in the category $\cM(P/U)$.

For the notion of a parabolic subgroup of a non-reductive group see {\it e.g.} \cite[\S 11]{BorGr}.

We will apply this theorem in the case when $P$ is a parabolic subgroup of a \DimaL{linear algebraic} group $G$. In this case $Q$ is a minimal parabolic subgroup of $G$ as well.
}
\section{Preliminaries on holonomic $D$-modules and distributions}\label{sec:Dmod}

We will now recall some facts and notions from the theory of $D$-modules on smooth  algebraic varieties. For the proofs and for further details we refer the reader to \cite{Ber,Bor,HTT}.

By a $D$-module on a smooth affine algebraic variety $\bf X$ we mean a module over the algebra $D(\bf X)$ of differential operators.
The algebra $D(\bf X)$ is equipped with a filtration, defined by the order of differential operators. This filtration is called the geometric filtration. The associated graded algebra with respect to
this filtration is the algebra $\cO(T^*{\bf X})$ of regular functions on the total space of the cotangent
bundle of ${\bf X}$. This implies that the algebra $D(\bf X)$ is Noetherian.

This allows us to define the singular support of a finitely generated $D$-module
$M$ on ${\bf X}$ in the following way. Choose a good filtration on $M$, {\it i.e.} a filtration compatible with the filtration on $D(\bf X)$ such that the
associated graded module is a finitely-generated module over $\cO(T^*\X)$.
 Define the singular
support $SS(M)$ to be the support of the associated graded module. One can show that the singular support does
not depend on the choice of a good filtration on $M$, and that a good filtration  exists if and only if $M$ is finitely generated.

The Bernstein inequality states that, for
any non-zero finitely generated $M$, we have $\dim SS(M) \geq  \dim X$. If the equality holds then
$M$ is called holonomic. This property is closed under submodules, quotients, extensions, \DimaJ{and tensor products}, and implies finite length.

For any \DimaD{\Fre}space $V$, the space $\Sc^*(X,V)$ of tempered $\DimaD{V^*}$-valued distributions has a  structure of a right $D({\bf X})$-module, given by $(\xi d)(f)=\xi(df)$, where $d\in D({\bf X})$, $X={\bf X}(\R)$, $\xi \in \Sc^*(X,V)$  and $f\in \Sc(X)$. A distribution $\xi\in \Sc^*(X,V)$ is called holonomic if the submodule $\xi D({\bf X})\subset \Sc^*(X,V)$ generated by $\xi$ is holonomic. Note  that if $\xi$ is holonomic then so is $\xi p$, for any polynomial $p$ on $\bf X$.

For any polynomial $p\in \cO(\bf X)$, the algebra $D({\bf X}_p)$ of differential operators on the basic open affine set $\X_{p}:=\{x \in \X \text{ with }p(x)\neq 0\}$ is isomorphic to the localization $D({\bf X})_p$, i.e. the algebra of fractions of the form $dp^{-i}$. In order to define the latter algebra one proves that the family $p^i$ satisfies Ore conditions. This follows from the next lemma.

\begin{lem}\label{lem:Ore}
For any index $n$ and any $d\in D({\bf X})$ there exists $d'\in D({\bf X})$ such that $p^n d'=dp^{n+\deg d}$.
\end{lem}
This lemma is proven by induction on $\deg d$.

\begin{thm}[See {\it e.g.} {\cite[\S VI.5.2 and Theorem VII.10.1]{Bor}}]\label{thm:OpenPush}
Let $p\in \cO(\bf X)$ and let $M$ be a holonomic module over $D({\bf X}_{p})$. Then $M$ is holonomic also as a module over $D({\bf X})$.
\end{thm}

In the case when ${\bf X}$ is an affine space, another natural filtration on $D({\bf X})$ is possible. This filtration is called the arithmetic filtration, or the Bernstein filtration. It leads to a different definition of singular support and thus could a priori lead to a different notion of a holonomic module.  However, the two definitions of holonomicity are equivalent, since both are equivalent to a certain homological property, see  \cite[\S V.2]{Bor}.

\DimaA{All the notions defined in this section can be naturally extended to non-affine algebraic varieties, and the statements above continue to hold. However, we will only use the notion of  holonomic distributions on non-affine varieties. Namely,  a distribution is holonomic if its restrictions to some (equivalently, any) affine cover are holonomic. We now allow $X$ to be non-affine.}

\begin{lemma}[{See {\it e.g.} {\cite[Theorem VI.7.11]{Bor} and \cite[Facts 2.3.8 and 2.3.9]{AGAMOT}}}]\label{lem:hol}
Let ${\bf Z\subset X}$ be a closed smooth subvariety,  let $\xi\in \Sc^*(Z,V)$, and let $\eta\in \Sc^*(X,V)$ be the extension of $\xi$ to $X$ by zero. Then
\begin{enumerate}[(i)]
\item $\eta$ is holonomic if and only if $\xi$ is holonomic.
\item \label{hol:group} Let an algebraic group $\G$ act transitively on $\bf Z$, and its $\R$-points $G$ act linearly on $V$. Suppose that $\xi$ is $G$-equivariant \DimaK{and $V$ is finite-dimensional}. Then $\xi$ is holonomic.
\end{enumerate}
\end{lemma}

\DimaJ{For any $N\in \R$ denote $\C_{> N}:=\{\lam \in \C \, \vert \, \re \lam > N\}$. 
For any $r\in \R$ denote by $\Sc^*_{r}(X,V)$ the space of analytic functions $\xi:\C_{>r}\to \Sc^*(X,V)$. Note that for $r'<r$ there is a natural embedding of $\Sc^*_{r'}(X,V)$ into $\Sc^*_{r}(X,V)$, and denote by $\Sc^*_{\Lambda}(X,V)$ the direct limit $\varinjlim_{r\in \R} \Sc_r^*(X,V)$. Thus we regard the functions $\xi_{\lam}$ and
$\xi'_{\lam}$ as defining the same element in $\Sc^*_{\Lambda}(X,V)$ if they agree in some right half plane.

\DimaM{We will say that $\xi:\C_{>r}\to \Sc^*(X,V)$ is a \emph{meromorphic family} if there exists an analytic function $\alp$ on $\C$ such that $\alp \xi$ is analytic.}

Let $p$ be a polynomial on $\bf X$ with real coefficients, which takes non-negative values on $X:={\bf X}(\R)$ and let $\xi \in \Sc^*(X,V)$. By Lemma \ref{lem:Plam}, we can define a family  $\xi p^{\lam}\in\Sc^*_{\Lambda}(X,V) $.

Our main tool is the following theorem, essentially proven in \cite{Ber}.
\begin{thm}\label{thm:Ber}
If $\xi$ is holonomic then the family $\xi_{\lam}\in \Sc^*_{\Lambda}(X\DimaI{,V})$  defined by $\xi_{\lam}=\xi p^{\lam}$ for $\lam \in  \C_{> N}$ for some $N>0$ has a meromorphic continuation to the entire complex plane.
 Moreover, all the distributions in the extended family and all the  coefficients of the Laurent expansion at any $\lam \in \C$ are holonomic.
\end{thm}}
\begin{proof}
\DimaI{Since  tempered distributions form a sheaf, and so do holonomic tempered distributions, we can assume that $\bf X$ is affine.
For the case when $\bf X$ is the affine space $\A^n$ and $V=\C$, the theorem is \cite[Corollary 4.6 and Proposition 4.2(3)]{Ber}. The same proof works for $\X = \A^n$ and general $V$. \DimaJ{More precisely, this case follows from \cite[Proposition 4.4 and proof of Theorem 4.3]{Ber}. Indeed, let $N_{\xi}:=\xi D(\X)\subset \Sc^*(X,V)$ be the $D(\X)$-module generated by $\xi$. Let $M_p$ be the holonomic $D(\X)$-module defined in \cite[Definition 2.1]{Ber}
and let $M$ be the holonomic $D(\X)$-module $M_p\otimes N_{\xi}$. Define a morphism of $D(\X)$-modules $\Psi:M\to \Sc^*_{\Lambda}(X,V)$ by $\Psi(p^{\lambda}\otimes \zeta):=\zeta p^{\lambda}$. Adapting \cite[Proposition 4.4]{Ber} to this case we obtain the theorem.}

For general affine $\X$,} consider a closed embedding $i:\X\into \A^n$ and extend $\xi$ to a distribution $\eta$ on $\R^n$  by $\eta(f):=\xi(f\circ i)$. Also, extend $p$ to a polynomial $\DimaM{q}$ on $\R^n$ with real coefficients. By Lemma \ref{lem:hol} $\eta$ is holonomic, and thus the family $\eta_{\lam}\in \Sc^*_{\Lambda}(\R^n\DimaI{,V})$  defined by $\eta_{\lam}=\eta (\DimaM{q}^2)^{\lam/2}$ has a meromorphic continuation to the entire complex plane.


Note that for any polynomial $\DimaM{f}\in \cO(\R^n)$ that vanishes on $X$ and $\lam\in \bC_{>0}$  we have $\eta_{\lam}\DimaM{f}=0$. Thus this holds for all $\lam \in \C$ and thus the family $\eta_{\lam}$ can be restricted to a family $\xi_\lam$ of distributions on $X$. The family $\eta_{\lam}$ and all the coefficients of the its Laurent expansion at any $\lam \in \C$ are holonomic, and by Lemma \ref{lem:hol} the same holds for the family $\xi_{\lam}$  and all its Laurent coefficients.
\end{proof}

\DimaM{
\begin{remark}
The obtained meromorphic family becomes analytic after multiplication by an appropriate $\Gamma$ factor.
\end{remark}
}


\section{Main results}\label{sec:main}
In this section we fix a local field $F$ of characteristic zero and an algebraic $F$-variety $\bf X$.
 For any \DimaA{regular function} $p\in \cO(\X)$ we denote by $\X_{p}$ the  open subset
$$\X_{p}:=\{x \in \X \text{ with }p(x)\neq 0\}. $$ 
\DimaK{Denote also $X_p:=\X_p\cap \X(\R)$}.
Let a linear  algebraic $F$-group $\mathbf{Q}$ act (algebraically) on $\bf X$. Recall the categories $\cC(Q)$ and $\cC^{\mathrm{fin}}(Q)$ from  Notation \ref{not:C2}. Let $(\sigma,V) \in \cC(Q)$.
%
%

\begin{lemma}\label{lem:key}
Assume that $F=\R$ and $\X$ is smooth.
 Let $p\in \cO(\X)$ be  real-valued on $X$ and let \DimaK{$U\subset X$ be a Zariski open $Q$-invariant subset that includes $X_{p}$.} Let $\xi\in \Sc^*(U,V)$ be holonomic. Let $\psi$ be an algebraic character of $Q$ and assume that
\DimaL{ for all $\lam\in \C$ with $\re \lam$ big enough, $\xi |p|^\lam$ is $(Q,|\psi|^\lam\sigma)$-equivariant.}


\DimaK{Then there exists a meromorphic $(Q,|\psi|^{\lam}\sigma)$-equivariant holonomic family  \DimaM{$\{\eta_{\lam}\in \Sc^*(X,V)\, ,\, \lam \in \C\}$} 
such that $\eta_{\lam}|_{X_p}=(\xi|_{X_p})|p|^{\lam}$ and the constant term $\eta$ of $\eta_{\lam}$ is generalized $Q$-equivariant and satisfies  $\eta|_{X_p}=\xi|_{X_p}$.} \end{lemma}

\begin{proof}
Replacing $p$ by $p^2$ if needed we assume that $p$ is non-negative on $X$, and thus $\psi$ is positive on $Q$.

\DimaI{Since $F=\R$ and by definition of $\cC(Q)$, $V$ is a nuclear \Fre space}.
By Corollary \ref{cor:ext}, $\xi$ extends to a $(Q,\sigma)$-equivariant functional $\xi'$ on
$$F^n(X,U)= \{\phi \in \Sc(X,V)| \, \phi\equiv 0 \text{ on }X
\setminus U \text{ with first $n$ derivatives} \},$$
for some $n$, such that for any $d\in D(\bf X)$ we have
\begin{equation}\label{=0}
\text{if } \xi d =0 \text{ then } \xi' (d(F^{n+\deg d}(X,U)))=0.
\end{equation}
Define $\eta_n \in \Sc^*(X,V)$ by $\eta_n(\phi):=\xi'(p^n\phi)$. Note that $\eta_n$ is $(Q,\psi^{n}\sigma)$ - equivariant.


Let us show that $\eta_n$  is a holonomic distribution.
\DimaA{For this purpose we can assume that $\X$ is affine}.
Let $I$ be the annihilator ideal of $\xi$ in $D({\bf X})$, i.e. the ideal of all $d \in D(\bf X)$ with $\xi(d|_{\bf U})=0$.
Let $d_1,\dots d_l$ be a finite set of generators of $I$
Let $J$ denote the annihilator ideal of $\eta_n$ in $D(\X)$. By Lemma \ref{lem:Ore}, for any $i\leq l$ we can find  $d_i'\in D(\bf X)$ such that  $d_ip^{n+\deg d_i}=p^nd_i'$.
 Then for any $\phi\in \Sc(X,V)$ we have
 $$(\eta_nd_i')(\phi)=\eta_n(d_i'(\phi))=\xi'(p^{n}d_i'(\phi))=\xi'(d_i(p^{n+\deg d_i}\phi)).$$
Since $p^{n+\deg d_i}f\in F^{n+\deg d_i}(X,U)$, from \eqref{=0} we have $d_i'\in J$. Thus the localization $J_p$ of $J$ includes $p^{-n}d_i$ for all $i$. Note that $\{p^{-n}d_i\}$ generate the ideal $p^{-n}ID({\bf U})$, which is the annihilator of $\xi p^n$ in $D({\bf U})$. Since $\xi p^n$ is holonomic, we get that $D({\bf U})/J_p$ is holonomic.
Now,   $D({\bf U})/J_p=\DimaJ{N}_p$, where $\DimaJ{N}:=\eta_nD(\bf X)$. Thus $\DimaJ{N}_p$ is holonomic, and Theorem \ref{thm:OpenPush} implies that $N$ is holonomic and thus so is $\eta_n$.

Consider the analytic family of equivariant distributions $\eta_{\lam}:=\eta_{n}|p|^{\lambda-n}$ defined for $\re{\lam}$ big enough.  It is easy to see that this family is  $(Q,\psi^{\lam}\sigma)$-equivariant.
By Theorem \ref{thm:Ber} the family $\eta_{\lambda}$ has a meromorphic continuation to the entire complex plane. Now, define $\eta$ to be the constant term of this family. Note that $\eta|_{X_{p}}=p^0\xi|_{X_p}=\xi|_{X_p}$. By Lemma \ref{lem:terms}, $\eta$ is generalized $(Q,\sigma)$-equivariant.
\end{proof}

\begin{rem}
The distribution $\eta$ gives rise to $Q$-equivariant distributions on $X$. However, the restrictions of these distributions to $U$ might vanish. {\it E.g.}  in the situation of Example \ref{ex:R}, all $Q$-invariant distributions on $X$ are supported at the origin. Note also that the temperedness assumption on $\xi$ is necessary \DimaJ{even for the trivial group action. {\it E. g.} the distribution $\exp(1/x)dx$ on $\R^{\times}$ does not extend to a distribution on $\R$.}
\end{rem}

\DimaC{Note that complex algebraic groups and varieties can be viewed as $\R$-points of algebraic groups and varieties defined over $\R$.}

Let us now formulate and prove our main results.

\begin{thm}\label{thm:GenA}
Assume that $\mathbf{Q}$ is quasi-elementary, and \DimaA{$\X$ is quasi-affine}.
Assume that there exists a $Q$-orbit $O\subset X$ that admits a tempered \DimaK{holonomic} $Q$-equivariant $\DimaD{V^*}$-valued measure $\mu$. \DimaC{If $F$ is Archimedean we assume in addition that $X$ is smooth.}
Then there exists a generalized $Q$-equivariant \DimaK{holonomic distribution} $\eta\in \Sc^*(X,V)$ and a Zariski open $Q$-invariant neighborhood $U$ of $O$ such that the restriction of $\eta$ to $U$ equals the extension of $\mu$ to $U$ by zero.
\end{thm}

%
\begin{proof}

\DimaA{Let $\bf Z$ be the Zariski closure of $O$ in $\X$.}
By Corollary \ref{cor:pExists}, there exists an algebraic character $\psi$ of $\bfQ$ and a non-zero $(\bfQ ,\psi)$-equivariant \DimaA{regular function} $q\in \cO({\bf Z})$ defined over $F$.

If $F$ is non-Archimedean then, by Theorem \ref{thm:padic}, $\mu$ can be extended to a generalized $Q$-equivariant $\DimaD{V^*}$-valued distribution on $Z:={\bf Z}(F)$. Extending this distribution by zero we obtain a generalized $Q$-equivariant $\eta\in \Sc^*(X,V)$.

Now assume that $F$ is Archimedean and $X$ is smooth.
\DimaK{
Let $U:=(X\setminus Z)\cup O$.
 Note that $O$ is closed in $U$ and extend $\mu$ by zero to $\xi \in \Sc^*(U,V)$.
 Lift $q$ to \DimaA{a regular function $p\in \cO(\X)$ defined over $F$}.
\DimaL{Then $\xi |p|^\lam $ is $(Q,|\psi|^\lam\sigma)$-equivariant for all $\lam$ with $\re \lam> 0$.
Indeed, for any $g\in Q$, $|\psi(g)|^\lam |p|^\lam-|gp|^\lam$ vanishes on $Z\cap U$, thus $\xi(|\psi(g)|^\lam |p|^\lam-|gp|^\lam)=0$ and thus $g(\xi |p|^\lam)=\sigma(g)|\psi(g)|^\lam\xi |p|^\lam$.}}
By Lemma \ref{lem:hol} $\xi$ is holonomic.
From Lemma \ref{lem:key} we obtain a generalized $(Q,\sigma)$-equivariant extension of $\xi$ to $X$.
\end{proof}
\DimaK{
\begin{example}\label{exm:ExtDist}
Let  $\bfQ=$ upper triangular $2\times 2$ invertible matrices, $\X=\A^2$, and let $O=F^{\times}\times \{0\}$. Then $Z=F\times \{0\}$, and $U=(F^2)^{\times}$.
We can take $p(x,y)=q(x)=x$. Then $q$ is $Q$-equivariant, but $p$ is not. Still, if we take $\mu$ to be a Haar measure on $F^{\times}$  and extend it by zero to $\xi \in \Sc^*(U)$ then \DimaL{$\xi |p|^\lam$ is $Q$-equivariant for any $\lam$ with $\re \lam > 0$ and extends naturally to $F^2$ for $\re \lam > 1$. This family also has meromorphic continuation to the whole complex plane. At $\lam=0$ the continued family has a simple pole, with the residue equal to the delta-function at the origin.}
\end{example}
\DimaE{The following theorem generalizes Theorem \ref{thm:A}.}
\begin{thm}\label{thm:GenA0}
Assume that $\mathbf{Q}$ is quasi-elementary and $\X$ is quasi-projective.
Assume that there exists a $Q$-orbit $O\subset X$ that admits a tempered $Q$-equivariant holonomic $V$-valued measure $\mu$. If $F$ is archimedean we assume in addition that $X$ is smooth.
Then $\Sc^*(X,V)^Q\neq 0$.
\end{thm}
\DimaA{
\begin{proof}
For quasi-affine $\X$, the theorem follows from Theorem \ref{thm:GenA}, since the existence of a non-zero generalized equivariant distribution implies the existence of a non-zero  equivariant distribution.

For general quasi-projective $\X$, there exist, by Proposition \ref{prop:QP}, an algebraic (finite-dimensional) $F$-representation $\mathbf{W}$ of $\mathbf{Q}$ and a (locally closed) $\bf Q$-equivariant embedding of $\X$ into the projective space $\mathbb{P}(\W)$. Let  $\X':=pr^{-1}(\X)\subset \W$, where $pr:\W \setminus \{0\}\to \bP(\W)$ is the natural projection. By Lemma \ref{lem:PrFrob} we have an isomorphism of $Q$-representations
\begin{equation}\label{=Frob2}
\Sc^*(X,V)\simeq \Sc^*(X',V)^{F^{\times}}.
\end{equation} 
Let $\bfQ ':=\bfQ \times {\bf GL_1}$. Note that $\bfQ '$ is quasi-elementary and  $\X '$ is $\bfQ '$-invariant. Let $\chi'$ be the trivial extension of $\chi$ to $Q'$. Let $O':=pr^{-1}(O)$ and let $\mu'$ be the $(Q',\chi')$-equivariant measure on $O$ corresponding to $\mu$.
Since $\X'$ is quasi-affine, we have  $\Sc^*(X',V)^{Q\times F^{\times}}\neq 0$. By \eqref{=Frob2} this implies $\Sc^*(X,V)^{Q}\neq 0$.
\end{proof}
}

\begin{remark}
\DimaC{The argument in the proof of Theorem \ref{thm:GenA0} almost shows that
 Theorem \ref{thm:GenA} extends to quasi-projective varieties, except that in the case of archimedean $F$ we cannot yet show that the extended distribution $\eta$ is tempered. Let us sketch the line of reasoning.
As in the proof of Theorem \ref{thm:GenA}, we define a meromorphic family of equivariant distributions on the quasi-affine variety $X'$. Since this family is $F^{\times}$-equivariant,}
\DimaB{it defines a family of generalized sections of a family of analytic vector bundles on $X$, and we take the constant term of this family. However, we will not know that this constant term is tempered, since we do not know that the family is tempered, since we do not yet have a \DimaC{theory} of Schwartz sections of a non-Nash bundle. One of our students, Ary Shaviv, is currently working on such a notion.}
\end{remark}

}

Let $\mathbf{G}$ be a \DimaD{Zariski connected} linear algebraic $F$-group. Let $\bfQ,\bfH\subset \G$  be $F$-subgroups such that $\bfQ$ is quasi-elementary.
Consider the action of $\bfQ\times \bfH$ on $\G$ given by left multiplication by $\bfQ$ and right multiplication by $\bfH$.
 \DimaE{The following theorem generalizes Theorem \ref{thm:B}.}

\begin{theorem}\label{thm:GenB}
Let $\mu$ be a tempered $Q\times H$-equivariant \DimaK{holonomic} $\DimaD{V^*}$-valued measure on a double coset $QgH\subset G$. \DimaC{Let $(\sigma,V)\in \cC(Q\times H)$.}


Then there exists a generalized $Q\times H$-equivariant  \DimaK{holonomic} $\eta\in \Sc^*(G,V)$ and a Zariski open $Q\times H$-invariant neighborhood $U$ of $QgH$ such that the restriction of $\eta$ to $U$ equals the extension of $\mu$ to $U$ by zero.

Moreover, the dimension of $\Sc^*(G,V)^{Q\times H}$ is at least  the number of $Q\times H$-double cosets in $({\bf Q}g{\bf H})(F)$ possessing non-zero $Q\times H$-equivariant $\DimaD{V^*}$-valued tempered measures.

\end{theorem}

\begin{proof}
By Corollary \ref{cor:pQH}, there exists an algebraic character $\psi$ of $\bfQ\times \bfH$ and a non-zero $(\bfQ\times \bfH,\psi)$-equivariant $F$-polynomial $q$ on the Zariski closure ${\bf Z}$ of the double coset $\bfQ g \bfH$ \DimaD{that vanishes outside $\bfQ g \bfH$}.
Extend $\mu$ by zero to a $\bfQ \times \bfH$-equivariant $\DimaD{V^*}$-valued measure on $\bfQ g \bfH(F)$, that we will also denote by $\mu$.

If $F$ is non-Archimedean then, by Theorem \ref{thm:padic}, $\mu$ can be extended to a generalized $Q\times H$-invariant $\DimaD{V^*}$-valued distribution on $Z:={\bf Z}(F)$. Extending this distribution by zero we obtain a generalized $Q\times H$-invariant distribution $\eta$ on $G$.

Now assume that $F$ is Archimedean. \DimaK{Let $U:=(G \setminus Z)\cup QgH$
and let $\xi\in \Sc^*(U,V)$ be the extension of $\mu$ by zero. Lift $q$ to an $F$-polynomial $p$ on $G$ and} \DimaL{note that for all $\lam$ with $\re \lam$ big enough, $\xi |p|^\lam$ is $(Q\times H,|\psi|^\lam\sigma)$-equivariant}. By Lemma \ref{lem:hol}, $\xi$ is holonomic.
 Applying Lemma \ref{lem:key} to $p,\xi, \G$ and $\bfQ g \bfH$, we obtain a generalized equivariant extension $\eta$ of $\xi$ to $\Sc^*(G,V)$.

To prove the ``moreover" part, let $g_1,\dots g_n\in  (\bfQ g \bfH)(F)$ such that the double cosets $Qg_iH$ are distinct and posses invariant measures $\mu_1,\dots \mu_n$. These measures extend by zero to linearly independent distributions $\xi_i$ on $U$.  Applying Lemma \ref{lem:key} we construct meromorphic families $\xi_{i,\lambda}$ of distributions on $G$. It is easy to see that the support of each family $\xi_{i,\lambda}$  lies in the closure of $Qg_iH$ and thus the families are linearly independent. Let $L$ denote the linear span of \DimaD{the Laurent expansions of these families at $\lambda=0$} and let $W\subset \Sc^*(G,V)$ denote the space spanned by the leading coefficients of the \DimaD{series} in $L$. By Lemma \ref{lem:Spec}, $\dim W=\dim L=n$ and by Lemma \ref{lem:terms} all the distributions in $W$ are $Q\times H$-equivariant.
\end{proof}


\begin{rem}
The lower bound on  dimension as in Theorem \ref{thm:GenB} can be shown to hold under the conditions of Theorem \ref{thm:GenA} as well.
\end{rem}

\begin{proof}[Proof of Corollary \ref{cor:Prin}]
\DimaD{By Lemma \ref{lem:Frob} 
the double coset $P_0gH$ has a non-zero tempered $P_0\times H$ equivariant $\sigma\Dima{^*}$-valued measure. By Theorem \ref{thm:B},
this implies that $\Sc^*(G,\sigma)^{P_0\times H}\neq 0$, and thus $\Sc^*(G,\sigma)^{P_0 \times H}\neq 0$. By Lemma \ref{lem:Integ}, $\Sc^*(G,\sigma)^{H\times P_0}$ is isomorphic to the space of $H$-invariant continuous functionals on $\Ind_{P_0}^G(\sigma)$.}
\end{proof}

\DimaM{
As we will show later, the following is a generalization of Proposition \ref{prop:ParAdapt}.

\begin{prop}
\label{prop:ParSpher} Let $\mathbf{G}$ be a linear algebraic $F$-group, let $\mathbf{H,P\subset G}$ be $F$-subgroups, let $C:=H\cap P$,  and let
$\sigma \in {\mathcal{M}}(P)$. Suppose that

\begin{enumerate}

\item $\mathbf{P}$ is a parabolic subgroup, and the complement $G\setminus HP
$ is the zero set of an $(H\times P,1\times \psi )$-equivariant $F$-polynomial $f$ on $G$.

\item \label{it:spher} either $\sigma$ is finite-dimensional or $\mathbf{H}$ has finitely many orbits on the flag variety of minimal
parabolic $F$-subgroups of $\mathbf{G}$.

\item $\sigma $ admits a non-zero $(C,\Delta _{C}\Delta _{H}^{-1})$
-equivariant continuous {linear} functional.
\end{enumerate}

Then $\Ind_{P}^{G}(\sigma )$ admits a non-zero $H$-invariant
{continuous linear} functional.
\end{prop}

\begin{proof}
 In view of Lemma 2.26, we need to find a
non-zero element
\begin{equation*}
\xi \in {\mathcal{S}}^{\ast }(G,\sigma \otimes \Delta _{P}^{-1})^{H\times P}.
\end{equation*}%
By Lemma \ref{lem:Frob}, the open double coset $HP$ has a tempered $H\times P
$ equivariant $(\sigma \otimes \Delta _{P}^{-1})^{\ast }$-valued measure $%
\mu $. If $F$ is non-archimedean then the non-vanishing of ${\mathcal{S}}^{\ast }(G,\sigma \otimes \Delta _{P}^{-1})^{H\times P}$ follows from Theorem \ref{thm:padic}.

From now on we assume that $F$ is archimedean. By Corollary \ref{cor:ext}, $\mu f^{2n}$ has a natural extension $\xi
_{n}$ to $G$ for $n$ big enough. Multiplying this further by $|f|^{{\lambda }%
-2n}$ we obtain a family of distributions%
\begin{equation*}
\xi _{\lambda }\in {\mathcal{S}}^{\ast }(G,\sigma \otimes \Delta
_{P}^{-1})^{H\times P,1\otimes \psi ^{{\lambda }}}\text{ for }Re\left(
\lambda \right) \gg 0.
\end{equation*}%
If $\sigma $ is finite dimensional then $\xi _{\lambda }$ is holonomic and
admits a meromorphic continuation in $\lambda $, and we obtain $\xi $ by
considering the principal part at $\lambda =0$.

From now on we assume that $\sigma $ is
infinite-dimensional and $\bf H$ satisfies the finiteness condition in \eqref{it:spher}. In this case $\xi _{\lambda }$ is not necessarily holonomic
but we will show that it is \emph{still possible} to obtain a meromorphic
continuation and obtain $\xi $ as before.

For this we first note that by Theorem \ref{thm:CasSubRep} we have an
epimorphism
\begin{equation*}
T:Ind_{Q}^{P}\rho \twoheadrightarrow \sigma .
\end{equation*}%
where $Q\subset P$ is a minimal parabolic subgroup and $(\rho ,V)\in
\mathcal{C}^{\mathrm{fin}}(Q)$. We will now construct a monomorphism%
\begin{equation*}
\Phi :{\mathcal{S}}^{\ast }(G,\sigma )\hookrightarrow {\mathcal{S}}^{\ast
}(G\times P,V)
\end{equation*}
with suitable equivariance under $G\times P$, and satisfying

\begin{enumerate}
\item[(a)] The family $\eta _{\lambda }=\Phi \left( \xi _{\lambda }\right) $
is holonomic and hence admits a meromorphic continuation.

\item[(b)] The continued family is in the image of $\Phi $ and so determines
a meromorphic continuation of $\xi _{\lambda }$.
\end{enumerate}

To construct $\Phi $, we fix a left Haar measure $dq$ on $Q$ and consider
the integration operator%
\begin{equation}\label{=I}
I:{\mathcal{S}}(P,V)\twoheadrightarrow \Ind_{Q}^{P}(\rho ,V),\quad
I\left( a\right) \left( p\right) :=\int_{Q}\rho \left( q\right) a\left(
pq\right) dq
\end{equation}%
The map $I$ is a surjection by \cite[Remark 2.1.4]{dCl}, and hence so is $%
\phi =T\circ I$. Setting $\tau =Ker\left( \phi \right) $ we obtain a short
exact sequence of representations of $P,$ which we regard as $G\times P$%
-modules with trivial $G$-action%
\begin{equation}
0\rightarrow \tau \rightarrow {\mathcal{S}}(P,V)\overset{\phi }{%
\twoheadrightarrow }\sigma \rightarrow 0.  \label{=RepP}
\end{equation}%
Tensoring this with ${\mathcal{S}}(G)$, equipped with the \emph{usual}
action of $G\times P$, we get a  sequence of $G\times P$-modules, which is  a short exact sequence by Lemma \ref{lem:flat}.
\begin{equation}
0\rightarrow {\mathcal{S}}(G)\,\widehat{\otimes }\,\tau \rightarrow {%
\mathcal{S}}(G)\,\widehat{\otimes }\,{\mathcal{S}}(P,V)\overset{1\otimes
\phi }{\twoheadrightarrow }{\mathcal{S}}(G)\,\widehat{\otimes }\,\sigma
\rightarrow 0.  \label{=IndG}
\end{equation}
Now identifying ${\mathcal{S}}(G)\,\widehat{\otimes }\,{\mathcal{S}}(P,V)$
with ${\mathcal{S}}(G\times P,V)$ and dualizing (\ref{=IndG}), we get
\begin{equation}
0\rightarrow {\mathcal{S}}^{\ast }(G,\sigma )\overset{\Phi }{\rightarrow }{%
\mathcal{S}}^{\ast }(G\times P,V)\rightarrow {\mathcal{S}}^{\ast }(G,\tau
)\rightarrow 0.  \label{=IndGDual}
\end{equation}

The map $\Phi $ is $G\times P$-equivariant for the induced action, and to
complete the proof it remains to show that $\eta _{\lambda }=\Phi \left( \xi
_{\lambda }\right) $ satisfies (a) and (b).

For (a) we note that the $G\times P$-equivariance of $\Phi $ implies that $%
\eta _{\lambda }$ is $\left( H\times P,1\otimes \psi ^{{\lambda }}\right) $%
-equivariant. We claim that $\eta _{\lambda }$ is also $\left( Q,\Delta
_{Q}^{-1}\right) $-equivariant with respect to the action of $Q$
on ${\mathcal{S}}(G\times P,V)$ given by%
\begin{equation*}
\left( q\cdot a\right) (g,p)=\rho (q)a(g,pq).
\end{equation*}%
For this we argue as follows. By \eqref{=I}, for any $a\in {\mathcal{S}}(G\times P,V)$ we have $I(q\cdot a)=\Delta_Q(q)a$ and thus $\Delta
_{Q}a-q\cdot a\in \Ker(I)$. Since $\eta_{\lam}\in\Im(T\circ I)^*,$ $\eta_{\lam}$ vanishes on $\Ker(I)$ and thus
$$\langle q\cdot \eta_{\lam} - \Delta
_{Q}^{-1}(q)\eta_{\lam},a\rangle=\langle \eta_{\lam},\Delta
_{Q}(q)a-q\cdot a\rangle=0$$ for any $a\in  {\mathcal{S}}(G\times P,V)$. Thus $q\cdot \eta_{\lam} = \Delta
_{Q}^{-1}\eta_{\lam}$.

Thus $\eta _{\lambda }$ is $\left( H\times P\times Q,1\otimes \psi ^{{\lambda}}\otimes \Delta
_{Q}^{-1}\right) $-equivariant. Since by assumption $\mathbf{H\times
P\times Q}$ has finitely many orbits on $\mathbf{G\times P}$ and $\rho $ is
finite dimensional, by Lemma \ref{lem:hol} $\eta _{{\lambda }}$ is holonomic.

To prove (b) we note that by (\ref{=IndG}) and (\ref{=IndGDual}) the image
of $\Phi $ can be characterized as%
\begin{equation*}
\Im\left( \Phi \right) =({\mathcal{S}}(G)\,\widehat{\otimes }\,\tau )^{\bot
}\subset {\mathcal{S}}^{\ast }(G\times P,V)
\end{equation*}
Thus we first have%
\begin{equation*}
\left\langle \eta _{\lambda },{\mathcal{S}}(G)\,\widehat{\otimes }\,\tau
\right\rangle =0\text{ for }Re\left( \lambda \right) \gg 0
\end{equation*}%
and after meromorphic continuation, this holds for $\lam$ in an open dense subset of $\mathbb{C}$.
Consequently the continued family $\eta _{\lambda }$ lies in the image of $%
\Phi $. This proves (b) and completes the proof of the Proposition.
\end{proof}
}

\DimaK{Note that assumption \eqref{it:spher} is weaker than absolute sphericity and stronger than $F$-sphericity.
\begin{remark}
In fact,  assumption \eqref{it:spher} can be replaced by the assumption that $\bf H \cap P$ has finitely many orbits on the flag variety of minimal parabolic $F$-subgroups of $\bf P$.
Indeed, in this case
$\bf H\times P\times Q$ has finitely many orbits on $\bf HP \times P$, and thus $\eta_{\lam}|_{PH\times P}$ is holonomic. Since for $\re \lam$ big enough, $\eta_{\lam}$ is the natural extension of $\eta_{\lam}|_{PH\times P}$ obtained by Corollary \ref{cor:ext}, we get that $\eta_{\lam}$ is also holonomic, and thus has a meromorphic continuation. We can further drop the condition that $\bf P$ is a parabolic subgroup, but then we will have to discuss functionals on the Schwartz induction (see \cite[Definition 2.1.3]{dCl}).
\end{remark}

\begin{proof}[Proof of Proposition \ref{prop:ParAdapt}]

By Corollary \ref{cor:pQH}, there exists a non-zero $F$-polynomial $q$ on $\G$ such that $\bfH\mathbf{P_0}=\G_q$. Since $\bfH\mathbf{P_0}=\bfH\mathbf{P}$, Corollary \ref{cor:act} implies that $q$ changes under the action of the connected component $(\bfH \times \mathbf{P} )^0$ of $\bfH \times \mathbf{P}$ by some (algebraic) character $\psi$. Replacing $q$ by the product of its shifts by representatives of all the connected components of $\bf H\times P$, we obtain a polynomial $p'$ on $\G$ that changes by some  character under the action of  $\bf H\times P$, and such that its set of non-zeros is still $\bfH\mathbf{P}$ . \DimaJ{Let $E/F$ be a finite Galois extension such that $p'$ is defined over $E$.} Replacing $p'$ by the product of its shifts under the Galois group $\mathrm{Gal}(E/F)$ we obtain an $\bf H\times P$-equivariant $F$-polynomial $p$ with the same set of zeros. The proposition follows now from Proposition \ref{prop:ParSpher}.
\end{proof}


It is shown in \cite{BD} that if $\G$ is reductive and $\bf H$ is a symmetric subgroup, then any $\theta$-split parabolic subgroup satisfies the conditions of the proposition, where $\theta$ is the involution that defines $\bf H$. One can show that if $\bf H\subset G$ is a horospherical subgroup, $\bf B\subset G$ is a Borel subgroup and $\bf D\subset G/H$ is a $\bf B$-stable divisor then the stabilizer of $\bf D$ in $\bf G$ also satisfies the conditions. In a subsequent paper we are going to find sufficient conditions on $\bf P$ for general spherical subgroups $\bf H\subset G$.


}
\DimaF{
\begin{remark}
We require our groups to be Zariski connected for two reasons. First, this holds in all the applications we have in mind and second, this is required in the statements from \cite{KK} that we use in \S \ref{subsec:GrAct}. However, if we assume the subgroup $\bf Q$ in \S \ref{subsec:GrAct} to be solvable (rather than quasi-elementary) then one can prove the corresponding statements directly without the connectivity assumption. As a result, the same holds regarding Theorems \ref{thm:A} and \ref{thm:B}: if $Q$ is solvable then it does not need to be connected.
\end{remark}
}

\subsection{\DimaA{An open question}}

The next natural question that arises is whether one can extend distributions which are supported on an orbit but not defined on this orbit.
\begin{ques}\label{ques}
Let a quasi-elementary group $\bfQ$ act on a smooth affine variety $\X$ defined over $\R$. Let $\bf U \subset X $ be a (Zariski) open $\bfQ$-invariant subset.
Let $\chi$ be a character of $Q$. Let $O\subset  U$ be a closed $Q$-orbit and assume that there exists $\xi\in \Sc^*(U,\chi)^{Q}$ with $\Supp \xi = O$. Does this imply $\Sc^*(X,\chi)^Q\neq 0$?
\end{ques}

\DimaB{It seems that the solution of this question requires a new method. We were able, however, to solve the following special case.}

\begin{example}\label{exm:ExtDistQues}
The answer to Question \ref{ques} is yes if  $\bfQ=$ upper triangular $2\times 2$ invertible matrices, $\X=\A^2$.
\end{example}
\begin{proof}
The only case that does not follow from Theorem \ref{thm:A} is
${\bf U}=\A^2\setminus 0, \, O=\R^{\times}\times \{0\}$. Fix $\xi\in \Sc^*(U,\chi)^{Q}$ with $\Supp \xi = O$.
If $y\xi=0$ then $\xi$ is a measure on $O$ and we use  Theorem \ref{thm:A} again.
Assume $\xi y^2=0, \, \xi y \neq 0$. Note that $\fq:=Lie(Q)$ acts on $X=\R^2$ by the vector fields $\alp=x\partial_x,\, \beta=y\partial_y, \, \gamma=y\partial_x.$
Thus the coordinate $y$ and the vector field $\partial_x$ are $Q$-equivariant.
Since $\xi$ is equivariant $\xi\partial_y y$ is proportional   to $\xi$ and thus is $(Q,\chi)$-equivariant. Now, we build the family $\xi y x^{\lam}$  as in Lemma \ref{lem:key}, take the leading coefficient of the Laurent expansion at $\lam=0$ and apply $\partial_y$ to obtain a $(Q,\chi)$-equivariant distribution.
If $\xi$ has order $n$ along $Z$, we consider the family  $\xi y^{n-1} x^{\lam}$ and apply $\partial_y^{n-1}$.
\end{proof}


\section{Application to generalized Whittaker spaces}\label{sec:GenWhit}
Let $G$ be a \DimaD{Zariski connected} real reductive group, $\fg$ be its Lie algebra, and $\kappa$ be the Killing form on $\g$.  For any nilpotent element $e\in \fg$, one defines a nilpotent subalgebra $\fr:=\fr_e\subset \fg$ such that $\kappa(e,[\fr,\fr])=0$ (see e.g. \cite[\S 2.5]{GGS}).
Then $e$ defines a character $\chi$ of $R:=\Exp(\fr)$ by $$\chi(\Exp(y)):=\Dima{\Exp(i\kappa(e,y))}.$$ For any smooth representation $\pi$ of $G$, one defines the generalized Whittaker space $\cW_e(\pi)$ to be the space of $(R,\chi)$-equivariant continuous  functionals on $\pi$. While the definition of $\fr_e$ involves some choices, for any two different choices the spaces of functionals are canonically isomorphic.

Now let $P\subset G$ be a parabolic subgroup, $N\subset P$ be its unipotent radical and $\fp,\fn$ be the Lie algebras of $P$ and $N$. Let $V$ be finite-dimensional (complex) representation of $P$, and let $\Ind_P^G(V)$ denote the smooth induction. The following theorem generalizes Theorem \ref{thm:C}.

\begin{thm}
Let $e\in \fn$. Let $\fr\subset \fg$ be a nilpotent Lie algebra with $\kappa(e,[\fr,\fr])=0$. Let $R:=\Exp(\fr)$ and  define a character $\chi$ of $R$ by $\chi(\Exp(y))=\kappa(e,y)$.
Then there exists a non-zero $(R,\chi)$-equivariant \DimaD{continuous linear} functional on $\Ind_P^G(V)$.
\end{thm}

\begin{proof}
By Lemma \ref{lem:Integ} we have  $$((\Ind_P^G(V))^*)^{R,\chi}=\Sc^*(G,\chi\otimes (V\Dima{^*}\otimes\Delta_P^{-1}))^{R\times P}.$$
Let us show that the double coset $RP$  has a tempered $R\times P$-equivariant $\chi\otimes (V\Dima{^*}\otimes\Delta_P^{-1})$-valued measure. By Lemmas \ref{lem:Frob} and \ref{lem:NilpTem}\eqref{it:tem} one has to show that $(\chi\otimes \Delta_R) \otimes V\Dima{^*} \otimes\Delta^{-1}_{C}$ has a $C$-invariant vector, where $C:=R\cap P$ is diagonally embedded into $R\times P$ and $\Delta_C$ denotes the modular function of $C$. Since $R$ is unipotent, so is $C=R\cap P$ and thus $\Delta_R=\Delta_{C}=1$. Also, $C$ lies in the unipotent radical of a minimal parabolic subgroup $P_0\subset P$. Applying Lemma \ref{lem:NilpTem}\eqref{it:nilp} to $P_0$ we obtain that $C$ acts unipotently on $V\Dima{^*}$ and thus has an invariant vector. Now,
$e$ lies in the nilradical of $\fp$, thus is orthogonal to $\fp$ under the Killing form and thus $\chi|_{C}=1$.
Altogether, we get that $(\chi\otimes \Delta_R) \otimes V \Dima{^*}\otimes\Delta^{-1}_{C}$ has a $C$-invariant vector. Thus $RP$  has a tempered $R\times P$-equivariant $\chi\otimes (V\Dima{^*}\otimes\Delta_P^{-1})$-valued measure, and by Theorem \ref{thm:B} we have $$\Sc^*(G,\chi\otimes (V\Dima{^*}\otimes\Delta_P^{-1}))^{R\times P}\neq 0.$$
\end{proof}

\end{document}